\documentclass{article}
\usepackage{amsmath, amssymb, amsthm, bbm, amsfonts}
\usepackage{enumitem}
\usepackage{tikz}
\usepackage[UKenglish]{isodate}
\usepackage[font=small,labelfont=it]{caption}
\usepackage{subcaption}
\usepackage{hyperref}
\usepackage{multirow}
\usepackage{fancyvrb}
\usepackage[figuresleft]{rotating}
\usepackage{geometry}

\usepackage{biblatex}
\bibliography{main}

\cleanlookdateon

\newcommand{\n}{\mathbbm{n}}

\makeatletter
\newcommand{\optionalitemlabel}[2]{%
  \phantomsection
  #1\protected@edef\@currentlabel{#1}\label{#2}%
}
\makeatother

\newtheorem*{theorem*}{Theorem}
\newtheorem{theorem}{Theorem}[section]

\newtheorem{proposition}[theorem]{Proposition}

\newtheorem{remark}[theorem]{Remark}

\newtheorem{example}[theorem]{Example}
\newtheorem{observation}[theorem]{Observation}

\title{Unimodal maps perturbed by heteroscedastic noise: An application to a financial systems}

\author{Fabrizio~Lillo\thanks{Dipartimento di Matematica, Universit\'a di Bologna and Scuola Normale Superiore, Pisa, Italy. Email address: fabrizio.lillo@unibo.it.} 
\and Giulia~Livieri\thanks{The London School of Economics and Political Science, London, United Kingdom, Email address: g.livieri@ac.lse.uk.} 
\and Stefano~Marmi\thanks{Scuola Normale Superiore, Pisa, Italy. Email address: stefano.marmi@sns.it.} 
\and Anton~Solomko\thanks{Envelop Risk, Bristol, UK. Email address: solomko.anton@gmail.com.}
\and Sandro~Vaienti\thanks{Aix Marseille Universit\'e, Universit\'e de Toulon, CNRS, CPT, 13009 Marseille, France. Email address: vaienti@cpt.univ-mrs.fr.}}

\date{\today}

\begin{document}
\maketitle

\begin{abstract}
We investigate and prove the mathematical properties of a general class
of one-dimensional unimodal smooth maps perturbed with a heteroscedastic noise. Specifically, we investigate the stability of the associated Markov chain, show the weak convergence of the unique stationary measure to the invariant measure of the map, and show that the average Lyapunov exponent depends continuously on the Markov chain parameters. Representing the Markov chain in terms of random transformation enables us to state and prove the Central Limit Theorem, the large deviation principle, and the Berry-Ess\'een inequality. We perform a multifractal analysis for the invariant and the stationary measures, and we prove Gumbel’s law for the Markov chain with an extreme index equal to 1. In addition, we present an example linked to the financial concept of systemic risk and leverage cycle, and we use the model to investigate the finite sample properties of our asymptotic results   
\end{abstract}

{\small
{\bf Keywords:} random dynamical systems, unimodal maps, Lyapunov exponents, leverage cycles, systemic risk.

{\bf 2020 Mathematics Subject Classification:} primary 91G80, 34F05; secondary 37H15, 62M45.
}

\newpage
\tableofcontents
\section{Introduction}
\indent    
In this paper, we investigate and prove some mathematical properties -- detailed below -- for the following discrete time dynamical system:  
\begin{equation}\label{eq:differenceequation}
    \phi_t = T(\phi_{t-1}) + \sigma_{\n}(\phi_{t-1})Y_{t-1}.
\end{equation}
Here $\phi_t$, $t \in \mathbb{N}_{\geq 1}$, is a sequence of real numbers in a bounded interval of $\mathbb{R}$, $T$ is a deterministic map on $[0,1]$ perturbed with the additive and heteroscedastic\footnote{A sequence of random variables is heteroscedastic if the variance is not constant.} noise $\sigma_{\n}(\phi_{t-1})Y_{t-1}$, being $\n \in \mathbb{N}_{\geq 1}$ a parameter that modulates the intensity of the noise; $\n$ is such that one retrieves the deterministic dynamic as $\n \rightarrow \infty$. Finally, $Y_t$, $t \in \mathbb{N}_{\geq 1}$, is a sequence of independent and identically distributed (\textit{i.i.d.}) real-valued random variables defined on some filtered probability space $(\Omega, \mathcal{F}, (\mathcal{F}_{t})_{t\geq 0}, \mathbb{P})$ satisfying to the usual conditions. The precise assumptions on $T$, $\sigma_{\n}$, and $Y_t$, $t \in \mathbb{N}_{>0}$, will be given in Section \ref{sec::assumptions}. The peculiarity of the model in Equation \eqref{eq:differenceequation} is that the law of the random perturbation, particularly its variance, depends on the position $\phi_{t-1}$ of the point, and therefore of its iterates by the dynamics. The model in \eqref{eq:differenceequation} can be used to describe situations where a slow deterministic dynamics interacts with a fast random one, and more generally when the two systems interact with a separation of time scale. For this reason, in the present paper, we put it in a very general setting. In Section \ref{sec::example}, we will present an example taken from a specific financial problem whose  dynamics can be brought back to \eqref{eq:differenceequation}.\\ 
 \indent    To study the mathematical properties of \eqref{eq:differenceequation}, we describe the dynamics of $\phi_{t}$, $t\in\mathbb{N}_{\geq 1}$, utilizing a Markov chain parametrized by $\n$; we will study the regime of finite $\n$ and the limit for $\n \rightarrow \infty$. As far as we know, the Markov chains with the kind of heteroscedastic noise we introduce are new (see \cite{gianetto2012estimating} for another type of heteroscedastic nonlinear auto-regressive process applied to financial time series). In particular, in the first part of the paper, we prove the following results. First, we investigate the stability of the Markov chain. It turns out that some specific properties of the stochastic kernel that defines our model do not allow us to apply general results available in the literature, as e.g., \cite{borovkov1998ergodicity} and \cite{meyn1993markov}. For instance, we do not know if our chain is Harris recurrent. We look, instead, at the spectral properties of the Markov operator associated with the chain on suitable Banach spaces and prove the quasi-compactness of such an operator. This result allows us to get finitely many stationary measures with bounded variation densities. The uniqueness of the stationary measure is achieved when the chain perturbs the map $T$, which is either topologically transitive on a compact subset of $[0,1]$ or an attracting periodic orbit. Second, we show the weak convergence of the unique stationary measure to the invariant measure of the map. This step is not trivial because the stochastic kernel becomes singular in the limit of large $\n$. Third, we introduce the average Lyapunov exponent by integrating the logarithm of the derivative of the map $T$ with respect to the stationary measure and show that the average Lyapunov exponent depends continuously on the Markov chain parameters. The previous result hinges on the \textit{explicit} construction of a sequence of random transformations close to $T$, which allows us to replace the deterministic orbit of $T$ with a random orbit given by the concatenation of the maps randomly chosen in the sequence. Notice that this construction is formally possible under general assumptions, but getting ``representations by special classes of transformations" is challenging, as Y. Kifer pointed out in \cite{kifier1986ergodic}. We believe this inference from the Markov chain to the random transformations is interesting and illustrates very well how the Markov chain randomly moves the states of the system. Representing the Markov chain in terms of random transformation enables us to state and prove some important limit theorems, such as the Central Limit Theorem, the large deviation principle, and the Berry-Ess\'een inequality. Fourth, for the class of unimodal maps $T$ of the chaotic type, we perform a multifractal analysis for the invariant and the stationary measures. Finally, we develop an \textit{Extreme Value Theory} (EVT, henceforth) for our Markov chain with finite values for the parameter $\n$. In particular, we prove Gumbel's law for the Markov chain with an extreme
index equal to 1. Notice that an EVT for Markov chains with the spectral techniques we will use is, as far as we know, a new result.\\
\indent     In the second part of the paper, we present an example linked to the financial concept of systemic risk to which our theory applies. In this setting, $\phi_t$ in \eqref{eq:differenceequation} represents the suitably scaled financial leverage of a representative investor (a bank) that invests in a risky asset. At each point in time, the scaling is a linear function of the leverage itself. The bank's risk management consists of two components. First, the bank uses past market data to estimate the future volatility (the risk) of its investment in the risky asset. Second, the bank uses the estimated volatility to set its desired leverage. However, the bank is allowed maximum leverage, which is a function of its perceived risk because of the Value-at-Risk (VaR) capital requirement policy. More specifically, the representative bank updates its expectation of risk at time intervals of unitary length, say $(t, t + 1]$ with $t\in\mathbb{N}_{\geq 1}$, and, accordingly, it makes new decisions about the leverage. Moreover, the model assumes that over the unitary time interval $(t, t + 1]$ the representative bank re-balances its portfolio to target the leverage without changing the risk expectations. The re-balancing takes place in $\n$
sub-intervals within $(t, t + 1]$.  In particular, the
considered model is a discrete-time slow-fast dynamical system; see \cite{dolgopyat2005averaging} and \cite{berglund2006noise}. After showing that the dynamics of the scaled leverage follows -- under suitable approximations -- a deterministic unimodal map on $[0, 1]$ perturbed with additive and heteroscedastic noise of the type of Equation (\ref{eq:differenceequation}), we perform a detailed numerical analysis in support of our theory. The numerical analysis also investigates the finite-size validity of some of our asymptotic results. In addition, we 
provide a financial discussion of the results. Notice that 
the example presented is a non-trivial extension of the models in \cite{corsi2016micro}, \cite{lillo2021analysis} and \cite{mazzarisi2019panic}, where the scaling of the leverage is constant. In particular, in \cite{lillo2021analysis}, the authors were also able to show that the constant-scaled leverage follows a (different) deterministic unimodal map with heteroscedastic noise. Also, they were able to prove the existence of a unique stationary density with bounded variation, the stochastic stability of the process, and the almost certain existence and continuity of the Lyapunov exponent for the stationary measure. In the present paper, we prove and extend the previous results but for a more general class of maps, and, as said, we generalize the model in \cite{lillo2021analysis}.\\
\indent \textit{Organization of the paper.} Section \ref{sec::assumptions} presents and discuss the working assumptions of the dynamics in \eqref{eq:differenceequation}. Section \ref{sec::markovchain} details the construction of the Markov chain. Section \ref{sec::randomtrasformation} represents our model regarding random transformations. In Section \ref{sec::mathematical_properties}, we investigate the mathematical properties of our model. An EVT theory for the Markov chain in \ref{sec::markovchain} is provided in Section \ref{sec::extremevalues}. In Section \ref{sec::example}, we present the financial model of a representative bank managing its leverage. We show that the model leads to a slow-fast deterministic random dynamical system which can be recast into a unimodal deterministic map with heteroscedastic noise of the type of Equation \eqref{eq:differenceequation}. We present and discuss some numerical investigation of this system in connection with our theory. 
\section{Assumptions}\label{sec::assumptions}
In this section, we define and discuss assumptions on $T$, $\sigma_{\n}$ and $Y_t$, $t \in \mathbb{N}_{\geq 1}$, as in Equation \eqref{eq:differenceequation}.  

\begin{enumerate}[label=(A\arabic*)]
\item\label{itm:A1} The map $T$ satisfies the following assumptions:
\begin{enumerate}
    \item $T$ is a continuous map of the unit interval $I\overset{\text{def}}{=}[0,1]$ with a unique maximum at the point $\mathfrak{c}$ such that $\Delta \overset{\text{def}}{=} T(\mathfrak{c}) < 1$.
    \item  There exists a closed interval $[d_1, d_2] \subset I$ which is forward invariant for the map and upon which $T$ and all its power $T^{t},\,t\in\mathbb{N}_{\geq 1}$ are topologically transitive\footnote{A map $T$ on a topological space $X$ is called \textit{topologically transitive} if for all nonempty open sets $U, V \subset X$ there exists $t$ such that $T^{-t} \cap V \neq \emptyset$. Notice that the topological transitivity of $T^{t}$, $t\in\mathbb{N}_{\geq 1}$ will be substantially used in Subsection \ref{subsec::lyapunovexponent}, \ref{subsec::limittheorems}, and \ref{subsec::multifractal}.}
    \item   $T$ preserves a unique Borel probability measure $\eta$, which is absolute continuous with respect to the Lebesgue measure. 
\end{enumerate} 
\end{enumerate}
Notice that Assumption \ref{itm:A1}--(b) is necessary in order to prove the mathematical properties in Section \ref{sec::mathematical_properties}; in general, one could ask for several transitive component but this would be an additional technicality that would not add to the present work’s conceptual advancements. Assumption \ref{itm:A1}-(c) is used only in the proof of the stochastic stability; see Subsection \ref{subsec::stationary}. 
We give now the following important
\begin{example}\label{ex::unimodalmaps}
An important class of maps susceptible to verify \ref{itm:A1} is given by the class of unimodal maps $T$ (\cite{thunberg2001periodicity} and \cite{melo2012onedimensional}) with negative Schwarzian derivative\footnote{The Schwarzian derivative $S(T)$ of the map $T$ is defined as $S(T):=\frac{T^{'''}}{T^{'}}-\frac{3}{2}\left(\frac{T^{'''}}{T^{'}}\right)^{2}$.}; notice that in this case, one has to require that the maps are at least $C^{3}$ on the interval $I$. Moreover, if $T$ verifies  $T(\Delta) < \mathfrak{c} < \Delta$, then the interval $[T(\Delta), \Delta]$, called \textit{dynamical core}, is mapped onto itself and absorbs all initial conditions; in particular $[d_1,d_2]$ in \ref{itm:A1}-(b) coincides with the dynamical core. The latter could exhibit motions other than simply attracting fixed points or 2-cycles. In the general class of unimodal maps $T$ with negative Schwarzian derivative, one can distinguish two types:
\begin{itemize}
    \item[(i)] $T$ is \textit{periodic} if there is a globally attracting fixed point or a globally attracting cycle.
    \item[(ii)] $T$ is \textit{chaotic} if \ref{itm:A1}-(b) and \ref{itm:A1}-(c) hold.
\end{itemize}
Perturbations of unimodal maps with uniform additive noise were studied
in \cite{baladi2002almost} and \cite{baladi1996strong}. As we already mentioned in the Introduction, \cite{lillo2021analysis}, instead, studies the perturbations of unimodal maps with heteroscedastic noise. 
\end{example}
\indent Before presenting the assumptions on $\sigma_{\n}$, we need to introduce the following quantities to which in the following we will refer:
\begin{equation}\label{eq:gap}
    \Gamma \overset{\text{def}}{=} 1 - \Delta,
\end{equation}
i.e., the gap between $T(\mathfrak{c})$ and 1. 
A positive constant $a$ satisfying one of the following two bounds:
\begin{equation}\label{eq:boundona1}
    a \leq \frac{1}{\sigma_{\text{max}}}\frac{\Gamma}{2},
\end{equation}
or 
\begin{equation}\label{eq:boundona2}
    a \leq \frac{1}{\sigma_{\text{max}}}\min\left\{\frac{\Gamma}{2}, \frac{q}{2}, \frac{1}{2}T\left(1-\frac{\Gamma}{2}\right)\right\},
\end{equation}
where $\sigma_{\max} \overset{\text{def}}{=} \max_{x \in I}\sigma_{\n}(x)$, and the positive constant $q$ is the eventual intercept of the map $T$ at zero (see Assumption \ref{itm:B12}).

\begin{enumerate}[label=(B\arabic*)]   
    \item\label{itm:B1} The function $\sigma_{\n}$ is a non-negative differentiable function for $x \in (0,1)$ such that $\forall x$ $\sigma_{\n}(x)\rightarrow 0$ as $\n \rightarrow \infty$. 
\end{enumerate}
We distinguish the following two sub-cases of  \ref{itm:B1}:
\begin{enumerate}[label=(B1.\arabic*)]
    \item\label{itm:B11}  $T(0)=0$, and so $\sigma_{\n}(0)=0$. In this case, we assume that for any fixed $\n \in \mathbb{N}_{>0}$ there exists $\varepsilon_{\n} \in \mathbb{R}$ such that:
    \begin{itemize}
        \item $T(x) - a \sigma_{\n}(x) > 0$ for $x \in (0, 1-\Gamma/2]$.
        \item $T(x) - a \sigma_{\n}(x) > x$ for $x \in (0, \varepsilon_{\n})$ (in particular, $T^{'}(0)>0$).
        \item $T(x) - a \sigma_{\n}(x) > \varepsilon_{\n}$ for $x \in (\varepsilon_{\n}, 1-\Gamma/2)$.
    \end{itemize}
    $\Gamma$ is defined in \eqref{eq:gap}, and $a$ satisfies \eqref{eq:boundona1}.
    \item\label{itm:B12}   $T(0)=q>0$. In this case, the positive multiplicative constant in \ref{itm:B11} satisfies \eqref{eq:boundona2}. 
\end{enumerate}
Assumptions \ref{itm:B11}-\ref{itm:B12} will allow us to define the transition probabilities for constructing our Markov chain. Indeed, the probability density $p_{\n}(x, \cdot)$ defining those probabilities will be supported on $[s_{a,-}(x), s_{a,+}(x)]$ with $s_{a,\pm}(x) = T(x)\pm a\sigma_{\n}(x)$; see Section \ref{sec::markovchain}. Moreover, \ref{itm:B11} requires $T$ to be $C^{1}$.\\ 
\indent Finally, we have that 
\begin{enumerate}[label=(C\arabic*)]
\item\label{itm:C1} $Y_t$, $t \in \mathbb{N}_{\geq 1}$ is a sequence of \textit{i.i.d} real-valued random variables defined on some filtered probability space $(\Omega, \mathcal{F}, (\mathcal{F}_{t})_{t\geq 0}, \mathbb{P})$ satisfying to the usual conditions. Their distribution function $g_{a}$, depending on the parameter $a$ in \ref{itm:B1}, is such that $\forall \omega \in \Omega$ and $x \in \tilde{I}$, with $\tilde{I} \supset I$, we have $T(x)+\sigma_{\n}(x)Y_{1} \in \tilde{I}$. The interval $\tilde{I}$ is slightly larger than $I$ and will be precisely determined later. Accordingly, the map $T$ will be extended on $\tilde{I}$. The distribution function $g_a$ has the following form:
\begin{equation}\label{eq:gaussiankernel}
    g_{a}(y)\overset{\text{def}}{=} c_{a} \chi_{a}(y)e^{-\frac{y^2}{2}},\quad y \in \mathbb{R},
\end{equation}
where $\chi$ is a $C^{\infty}$ bump function on $[-a,a]$ and
\begin{equation*}
    c_a = \left(\int_{\mathbb{R}} \chi_{a}(y)e^{-\frac{\varepsilon^2}{2}}\,dy\right)^{-1}.
\end{equation*}
\end{enumerate}

\noindent Assumption \ref{itm:C1} has two main objectives. On the one hand, the perturbation should not be too strong so that $T$ admits an extension to some compact interval $\tilde{I} \supset I$. On the other hand, the stochastic kernel associated with our Markov chain must be uniformly bounded on some interval to prove the Markov operator's quasi-compactness. As the introduction states, the Markov operator's quasi-compactness will provide stationary measures for the chain.\\

\noindent From now on, we will denote by 
\begin{equation}\label{eq::randomsystem}
    \mathcal{S}_{\n} \overset{\text{def}}{=} (T, \sigma_{\n}, Y)
\end{equation}
the triple composed by a map $T$ satisfying \eqref{itm:A1}, perturbed with an additive heteroscedastic noise in which the variance-like function $\sigma_{\n}$ and the noise verifies \eqref{itm:B1} and \eqref{itm:C1}, respectively.\\ \\

\indent Under \ref{itm:A1}, \ref{itm:B1}, and \ref{itm:C1}, we define the following stochastic process:
\begin{equation}\label{eq::stochasticprocess}
    F_t^{(\n)}(x) = T(x) + \sigma_{\n}(x)Y_{t},\,\,t\in\mathbb{N}_{\geq 1},\,x \in \tilde{I}.
\end{equation}
In addition, the random orbit associated with our initial difference equation is given by:
\begin{equation}\label{eq::randomorbit}
    F_{Y}^{t, (\n)}(x) = F_t^{(\n)} \circ F_{t-1}^{(\n)} \circ \ldots \circ F_{1}^{(\n)}(x),\,\,x \in \tilde{I}.
\end{equation}

Before proceeding, the following observation is in order. In Sections \ref{sec::markovchain}, we will see that several results valid under Assumption \ref{itm:A1} could also be extended for the class of maps in Example \ref{ex::unimodalmaps} that are periodic. This will be in particular relevant for the leading example in Section \ref{sec::example}. The validity of \ref{itm:A1} is much easier to verify for uniformly or even piecewise continuous maps. In principle, one could also consider multimodal maps as having several critical points. However, in the latter case, one has to handle the construction of the stationary measure as outlined in Section \ref{sec::mathematical_properties}. Notice that such a construction is also based on Assumption \ref{itm:B1} and \ref{itm:C1}.

\section{Markov chain}\label{sec::markovchain}
In this section, we define a Markov chain that describes our model. We obtain it as a deterministic map $T$ satisfying Assumption \ref{itm:A1} perturbed with an additive noise as in Assumptions \ref{itm:B1}-\ref{itm:C1}. In particular, for fixed $T$ we parametrize the chain by the \textit{intensity of the noise} $\n$, consequently indexing with $\n$ the chain $(X_t^{(\n)})$, the transition probabilities $P_{x}^{(\n)}$, and the stochastic kernel $p_{\n}(x, y)$. According to the theory of random transformations, a Markov chain can be constructed as follows; see, e.g., \cite{kifier1986ergodic}. Take an initial point $x \in I$ \footnote{One could also consider the initial point as a random variable $X_0$ independent of the $Y_t$; in this case, the measurable and bounded initial distribution is $\rho_0(A)=\mathbb{P}(X_0 \in A )$.} and define the following stochastic process for any $t \in \mathbb{N}_{\geq 1}$: 
\begin{equation}\label{eq:stocasticprocess}
    X_{t+1}^{(\n)} = F_{t+1}^{(\n)}(X_{t}^{\n}).
\end{equation}
Then, for $x \in I$ the transition probabilities are defined as:
\begin{equation}\label{eq:transitionprobabilities}
    P_{x}^{(\n)}(A) = \mathbb{P}(X_{t+1}^{(\n)} \in A | X_t^{(\n)}=x) = \mathbb{P}(F_{t+1}^{\n}(x)\in A) = \mathbb{P}(F_1^{(\n)}\in A),
\end{equation}
because all the $F_t^{(\n)}$, $t \in \mathbb{N}_{\geq 1}$, have the same distribution. By Assumption \ref{itm:C1}, and $\forall x\,:\,\sigma_{\n}(x)>0$, we have:
\begin{equation}\label{eq:stocastickernel}
\begin{split}
    P_{x}^{(\n)}(A) &= \int_{\mathbb{R}} 1_{A}(F_1(x))g_{a}(y)\,dy = \int_{\mathbb{R}}   1_{A}(T(x)+\sigma_{\n}(x)y)g_{a}(y)\,dy\\
                    &= \int_{\mathbb{R}} 1_{A}(z)\frac{1}{\sigma_{\n}(x)}g_{a}\left(\frac{z-T(x)}{\sigma_{\n}(x)}\right)\,dz \overset{\text{def}}{=} \int_{A} p_{\n}(x,z)\,dz,
\end{split}                  
\end{equation}
where $p_{\n}(x,y)$ is the stochastic kernel and in the third equality, we use the following change of variable: $z=T(x)+\sigma_{\n}(x)y$. Instead, if for some $x$ we have $\sigma_{\n}(x)=0$, the transition probability verifies $P_{x}(A)=1_{A}(T(x))$ (meaning $P_{x} = \delta_{T(x)}$, where $\delta_{T(x)}$ is the Dirac mass at $x$). So, the stochastic kernel is
\begin{equation}\label{eq::stochastickernel}
    p_{\n}(x,y) = \frac{1}{\sigma_{\n}(x)}g_{a}\left(\frac{z-T(x)}{\sigma_{\n}(x)}\right) = \frac{1}{\sigma_{\n}(x)}\chi_{a}\left(\frac{z-T(x)}{\sigma_{\n}(x)}\right)e^{-\frac{(z-T(x))^2}{\sigma_{\n}^2(x)}},
\end{equation}
with $\int p_{\n}(x,y)\,dy=1$ for every $x \in I, \sigma_{\n}(x)>0$. Therefore, $z \in [s_{a,-}, s_{a,+}]$ with $s_{a,\pm}=T(x) \pm a \sigma_{\n}(x)$.\\
\indent Since the noise varies in a neighborhood of 0, we need to enlarge the domain of definition of the map $T$ to take into account the action of the noise. More precisely, we extend the domain of $T$ to the larger interval $\tilde{I} \overset{\text{def}}{=} [-\Gamma, 1]$. On the interval $[-\Gamma,0]$, $T$ is extended continuously and decreasing with $T(-\Gamma) < \Delta$ and with the same slope of $T$ restricted to the interval $[0,\varepsilon_{\n}]$, where $\varepsilon_{\n}$ is given in \ref{itm:B11}.  With abuse of language and notation, we will continue to call $T$ the map after its redefinition, and we put $I = \tilde{I}$. We have the following remark.
\begin{remark}
\cite{baladi1996strong} consider a similar extension to allow perturbations with additive noise; in particular, it was supposed that $T$ admits an extension to some compact interval $J \supset I$, preserving all the previous properties and satisfying $T(\partial J) \subset \partial J$. Notice that, in our case and with these extensions, the map $T$ could lose smoothness in 0. However, this regularity persists on the interval $(0, 1)$, and this will be enough
for the subsequent considerations, particularly for the construction of the stationary measure whose support will be strictly included in $(0, 1)$.
\end{remark}

We look at the \textit{Markov operator} corresponding to the transition probabilities. To this aim, we denote by $\mathcal{M}$ the space of (real-valued) Radon measure on $\tilde{I}$, and by $\mathcal{L}\,:\,\mathcal{M}\rightarrow\mathcal{M}$ the Markov operator acting by 
\begin{equation*}
    \mathcal{L}\rho(A) = \int_{\mathbb{R}}  P_{x}^{(\n)}(A)\,d\rho(x),\quad \rho \in \mathcal{M},
\end{equation*}
for every Borel set $A \in I$, or, equivalently,
\begin{equation*}
    \int_{\mathbb{R}}  \varphi\,d\mathcal{L}\rho = \int_{\mathbb{R}} \int_{\mathbb{R}}  \varphi(y)dP_{x}^{(\n)}(y)d\rho(x)
\end{equation*}
for all $\varphi \in C^{0}$, where $C^{0}$ denotes the Banach space of continuous function on $I$ with the sup norm. In our case $\sigma_{\n}(\tilde{x})=0$ in at most two points, $\tilde{x}=0,1.$ Therefore in such a case we could write 
$$
\int \varphi d\mathcal{L}\rho =\iint_{\mathbb{R}\times \{\{0\}\cup\{1\}\}}\ \varphi(y) dP_x(y) d\rho(x)+\iint_{\mathbb{R}\times \{\mathbb{R}/\{0\}\cup \{1\}\}}\varphi(y) dP_x(y) d\rho(x)=
$$
$$
[\phi(T(0))\rho(\{0\})+\phi(T(1))\rho(\{1\}])+\iint_{\mathbb{R}\times \{\mathbb{R}/\{0\}\cup\{1\}\}}\varphi(y)p_{\n}(x,y)dyd\rho(x).
$$
We note that $\mathcal{L}:L^{1} \rightarrow L^{1}$ is an isometry, where $L^{1}$ is intended, from now on, with respect to the Lebesgue measure. In Subsection \ref{subsec::stationary}, we will be interested in stationary measures $\rho$, which are absolutely continuous with respect to the Lebesgue measure and, therefore, non-atomic. If we denote by $h \in L^{1}$ the density of such a measure, it will be a fixed point of the operator $\mathcal{L}\,:L^{1} \rightarrow L^{1}$, i.e., 
\begin{equation}\label{eq::fixedpoint}
    (\mathcal{L} h)(y) = \int_{\mathbb{R}}  p_{\n}(x, y)h(x)\,dx,\quad h \in L^{1},
\end{equation}
where $p_{\n}$ is the stochastic kernel in formula \eqref{eq::stochastickernel}. In particular, it should satisfy
\begin{equation}\label{eq::fixedpoint1}
    h(z) = \int_{\mathbb{R}}  p_{\n}(x, z)h(x)\,dx.
\end{equation}
We will return to the previous formula in Section \ref{sec::mathematical_properties}. Now, in the next section, we present a slightly different, yet equivalent (see, e.g., \cite{kifier1986ergodic}), approach for representing the model in Equation \eqref{eq:differenceequation}, namely the random transformation approach.
\section{Random transformations}\label{sec::randomtrasformation}
We consider the following identity:
\begin{equation}\label{eq::randomtransformations}
    T_{\eta}(x)=T(x)+\sigma_{\n}(x)\eta,~~~~~\,\eta \in [-a,a].
\end{equation}
Assumption \ref{itm:C1} implies that $T_{\eta}$ can be seen as a family of random maps of $I$ into itself. Let $\theta(\eta) \overset{\text{def}}{=} g_{a}(\eta)\,d\eta$ be the probability measure of $\eta$ with density $g_{a}$. Now, let $\rho \in \mathcal{M}$ (see Section \ref{sec::markovchain} for the definition of $\mathcal{M}$) a measure with density $h \in L^{1}$. By requiring its invariance, we have that:
\begin{equation*}
    \mathcal{L}\rho(A) = \rho(A) = \int_{\mathbb{R}} P_{x}^{(\n)}(A)\,d\rho(A).
\end{equation*}
In addition, by using the definition of $\rho(A)$, we have 
\begin{equation*}
    \begin{split}
        \rho(A) &= \int_{\mathbb{R}} 1_{A}(x)h(x)\,dx  = \int_{\mathbb{R}} \int_{\mathbb{R}} 1_{A}(T_{\eta}(x))h(x)\,dx\,d\theta(\eta)\\
                &= \int_{\mathbb{R}} d\theta(\eta)\int_{\mathbb{R}} 1_{A}(T_{\eta}(x))h(x)\,dx =\int_{\mathbb{R}}d\theta(\eta)\int_{\mathbb{R}} 1_{A}(x)\mathcal{L}_{\eta}h(x)\,dx,
    \end{split}
\end{equation*}
where $\mathcal{L}_{\eta}:L^{1}\rightarrow L^{1}$ is the Perron-Fr\"obenius operator associated to the map $T_{\eta}$\footnote{The Perron-Fr\"obenius operator associated to the map $T_{\eta}$ is defined by the duality relationship $\int_{\mathbb{R}}\mathcal{L}_{\eta}h\,g\,dx=\int_{\mathbb{R}}hg \circ T_{\eta}(x)\,dx$, where $h \in L^{1}$ and $g \in L^{\infty}$.}. By changing the order of integration again, we finally get that the Markov operator in Equation \eqref{eq::fixedpoint} satisfies for any $h \in L^1$ the following identity
\begin{equation}\label{eq::markovLepsilon}
    (\mathcal{L}h)(x) = \int (\mathcal{L}_{\eta}h)(x)\,d\theta(\eta).
\end{equation}

\noindent We now present a correlation integral that we will use to derive some statistical properties of our model. In order to do this, let $(\eta_t)_{t\geq 1}$ be an \textit{i.i.d.} stochastic process where each $\eta_t$ has distribution $\theta$, $\bar{\eta}_t \overset{\text{def}}{=} (\eta_{1}, \ldots, \eta_{t})$, and $\theta^{t}(\bar{\eta}_t)\overset{\text{def}}{=}\theta(\eta_1)\times\ldots\times\theta(\eta_t)$ the product measure. We call the following concatenation, or composition, of randomly chosen maps $T_{\eta_t} \circ\ldots\circ T_{\eta_1}$, where $(\eta_s)_{s=1}^{t}$ are \textit{i.i.d.} with distribution $\theta$, as \textit{random transformation}. In particular, the above-mentioned correlation integral reads as   
\begin{equation}\label{eq::correlationintegral}
    \int_{\mathbb{R}}(\mathcal{L}^{t}h)(x)g(x)\,dx=\int_{\mathbb{R}}\int_{\mathbb{R}}h(x)g(T_{\eta_1}\circ\ldots\circ T_{\eta_t})(x)d\theta^{t}(\bar{\eta}_t)\,dx,
\end{equation}
where $h \in L^{1}$ and $g \in L^{\infty}$. Notice that in \cite{lillo2021analysis}, authors use the Lebesgue measure instead of the probability measure $\theta$. By using the latter, we do not need to modify the map $T$ as in the Lebesgue measure case.

\section{Mathematical properties of the model}\label{sec::mathematical_properties}
In this section, we investigate some mathematical properties of our model. In Subsection \ref{subsec::stationary}, we show the existence and uniqueness of an absolutely continuous stationary measure and establish its convergence to the invariant measure of the deterministic map. This result allows us to define the Lyapunov exponent and prove its continuity with respect to the model parameters. We also discuss some limit theorems in Subsection \ref{subsec::limittheorems}. Finally, Subsection \ref{subsec::multifractal} concerns a multi-fractal analysis of our model.
\subsection{Stationary Measure and Stochastic Stability}\label{subsec::stationary} 
In this subsection, we establish the existence of a unique stationary measure for the Markov chain associated with our model.\\
\indent In Section \ref{sec::markovchain}, we extended the domain of definition of the map $T$ to the set $\tilde{I}= [-\Gamma, 1]$. In particular, if the constant $a$ satisfies the bound in Equation \eqref{eq:boundona2}, then the support of the stationary measure $\mu_{\n}$, say $\text{supp}(\mu_{\n}) \subset I_{\Gamma}$, where 
\begin{equation}\label{eq::supportone}
    I_{\Gamma}\overset{def}{=}\left[\frac{1}{2}T\left(1-\frac{\Gamma}{2}\right), 1-\frac{\Gamma}{2}\right].
\end{equation}
Indeed, on the one hand, if we take a point $z \in \left(1-\frac{\Gamma}{2}, 1\right]$, then it will be surely greater than $T(x) \pm a\sigma_{\n}(x)$, $x \in I$. In order to understand the left-hand side of the interval in Equation \eqref{eq::supportone}, suppose first that $T(0)=q>T\left(1-\frac{\Gamma}{2}\right)$. If $z \in \text{supp}(h)$, being $h$ the density of $\mu_{\n}$, then $T(x)\in [\tilde{s}_{a,-},\tilde{s}_{a,+}]$ with $\tilde{s}_{a,\pm}=z \pm a\sigma_{\n}(x)$, where $x\in\text{supp}(\mu_{\n})$ too. If $z$ is in a neighborhood of 0, then the values of $x$, which could contribute in $T(x)$ are smaller than $\left(1-\frac{\Gamma}{2}\right)$ by choice of $a$. So, if we take $z < \frac{1}{2}T\left(1-\frac{\Gamma}{2}\right)$ and we require that $a\sigma_{\n}(x)< \frac{1}{2}T\left(1-\frac{\Gamma}{2}\right)$, then $z\notin\text{supp}(h)$.\\
\indent If, instead, the constant $a$ satisfies Equation \eqref{eq:boundona1}, then the interval $I_{\varepsilon_{\n},\Gamma}\overset{def}{=}\left[\varepsilon_{\n},1-\frac{\Gamma}{2}\right]$ is invariant for $T_{\eta}$, $\forall \eta \in [-a,a]$ (see Equation \eqref{eq::randomtransformations}). In particular, if $x \in I_{\varepsilon_{\n},\Gamma}^{c}$, where $I_{\varepsilon_{\n},\Gamma}^{c}$ is the complementary set of $I_{\varepsilon_{\n},\Gamma}$, then $x$ will spend finitely many times in  $I_{\varepsilon_{\n},\Gamma}^{c}$; note that $x=0$ is a fixed point. In particular, the chain $X_t^{(\n)}$ visits finitely many times any open set $K$ in $I_{\varepsilon_{\n},\Gamma}^{c}$. Therefore, the chain is not recurrent and $\mu_{\n}(K)=0$.\\
\indent The above considerations implies that the subspace $\{h \in L^{1}\,:\,\text{supp}(h) \subset I_{\Gamma}\}$ (resp. $\{h \in L^{1}\,:\,\text{supp}(h) \subset I_{\varepsilon_{\n},\Gamma}\}$) is $\mathcal{L}$-invariant, and that the stochastic kernel $p_{\n}(x,z)$ has total variation of order $\frac{1}{\sigma_{\n}(x)}$. Therefore, it is \textit{uniformly bounded}\footnote{In general, we say that a stochastic kernel $p(x,y)$ has \textit{uniformly bounded variation} if $|p(x,\cdot)|_{TV} \in L^{\infty}$, i.e., there is $C>0$ such that $|p(x,\cdot)|_{TV}\leq C$ for almost every $x \in I$.} when restricted to $I_{\Gamma}\times I_{\Gamma}$ (resp. $I_{\varepsilon_{\n},\Gamma}\times I_{\varepsilon_{\n},\Gamma}$). In particular, we can apply Proposition 4.2 and Theorem 4.3 in \cite{lillo2021analysis} to conclude that the following proposition hold.
\begin{proposition}\label{prop::uniqueness}
    The random system in Equation \eqref{eq::randomsystem} admits a unique stationary measure $\mu_{\n}$ with density $h_{\n}$ of bounded variation and such that $[d_1, d_2] \subset \text{supp}(h_{\n})$. Moreover, for any observable $f \in L^{1}$, $g \in BV$, there exists $0<r<1$ and $C>0$, depending only on the system, such
that, for all $t \in \mathbb{N}_{\geq 0}$, we have
    \begin{equation*}
    \Big|\int_{\mathbb{R}}\int_{\mathbb{R}}(\mathcal{L}^{t}f)(x)g(x)\,dx-\int_{\mathbb{R}}f\,d\mu_{\n}\int_{\mathbb{R}}g(x)\,dx\Big|\leq C r^{t} \|f\|_{1}\|g\|_{BV}.
        \end{equation*}
\end{proposition}
\begin{proof}
Let $BV$ the Banach space of bounded variation functions on $I_{\Gamma}$ (or $I_{\Gamma,\varepsilon_{\n}}$) equipped with the complete norm 
\begin{equation*}
    \|f\|_{BV}=|f|_{TV}+\|f\|_1,
\end{equation*}
where $|f|_{TV}$ is the total variation of the function $f \in L^{1}$. Because the stochastic kernel has uniformly bounded variation on $I_{\Gamma}\times I_{\Gamma}$ (or on $I_{\Gamma,\varepsilon_{\n}} \times I_{\Gamma,\varepsilon_{\n}}$), we have 
\begin{equation*}
    \|\mathcal{L}\rho\|_{TV} \leq C\|\rho\|_{1}\quad\text{and}\quad\|\mathcal{L}\rho\|_{BV}\leq (C+1)\|\rho\|_{1};
\end{equation*}
see Lemma 4.1 in \cite{lillo2021analysis}. By the previous equation, we have 
\begin{equation*}
    \|\mathcal{L}\rho\|_{BV} \leq (C+1)\|\rho\|_{1}\leq \eta \|\rho\|_{BV} + (C+1)\|\rho\|_{1}
\end{equation*}
for any $\eta<1$; this is the \textit{Lasota-Yorke inequality} for the operator $\mathcal{L}$. The latter, plus the fact that $BV$ is compactly embedded in $L^1$, implies that the operator $\mathcal{L}$ has the following spectral decomposition
\begin{equation*}
    \mathcal{L} = \sum_{i} v_i \Pi_{i} + Q, 
\end{equation*}
where all $v_i$ are eigenvalues of $\mathcal{L}$ of modulus 1, $\Pi_{i}$ are finite-rank projectors onto the associated eigenspaces, $Q$ is a bounded operator with a spectral radius strictly less than one. They satisfy the following properties:  
\begin{equation*}
    \Pi_{i}\Pi_{j}=\delta_{i j}\Pi_{i},\quad Q\Pi_{i}=\Pi_{i}Q=0.
\end{equation*}
Standard techniques show that $1$ is an eigenvalue and therefore the chain will admit finitely many absolutely continuous ergodic stationary measures, with supports that are mutually disjoint up to sets of zero Lebesgue measure. Moreover the peripheral spectrum is completely cyclic. We require that 1 is a simple eingenvalue of $\mathcal{L}$, and that there is no other peripheral eigenvalue, hence implying that our Markov chain is mixing and therefore the norm of $\|\mathcal{L}^{t} f\|_{BV}$ goes exponentially fast to zero when $t\rightarrow\infty$, for $f \in BV$ and $\int_{\mathbb{R}} f\,dx=0$ (exponential decay of correlations).
These properties, which are consequences of the Ionescu-Tulcea-Marinescu theorem, are summarized by saying that the operator $\mathcal{L}$ acting on $BV$ is quasi-compact, see, e.g., \cite{hennion2001limit}; we will implicitly assume in the following that the operator has the mixing property too. In order to prove that 1 is a simple eigenvalue of $\mathcal{L}$, and that there is no other peripheral eigenvalue we first observe that the peripheral spectrum of $\mathcal{L}$ consists of a finite union of finite cyclic groups; then there exists $t \in \mathbb{N}_{\geq 1}$ such that 1 is the unique peripheral eigenvalue of $\mathcal{L}^{t}$.  It suffices then to show that the corresponding eigenspace is one-dimensional. Standard arguments show there exists a basis of positive eigenvectors for this subspace, with disjoint supports.
At this point we use 
a simple generalization of Theorem 4.3 in \cite{lillo2021analysis} for the powers of $\mathcal{L}^n$ plus the assumption on the topological transitivity of $T^n, n\ge 1$ on $[d_1, d_2]$  to get that the basis is one dimensional.

\end{proof}


\indent We investigate now the stochastic stability of the system, which means to determine if a sequence of stationary measure will converge \textit{weakly}\footnote{Notice that this result could be strengthened by showing that $\|h_{\n}-h\|_{1}\rightarrow 0$, which is called the \textit{strong stochastic stability}.} to the invariant measure of the unperturbed map. In our case, the sequence of probability measure is given by $\mu_{\n}\overset{\text{def}}{=}h_{\n}\,dx$. Notice that $h_{\n} \in L^{\infty},\,\forall\,\n$ because they have finite total variation. Nonetheless, to prove the above-mentioned stochastic stability, we need the following assumption  
\begin{enumerate}[label=(A$_{p}$)]   
    \item\label{itm:Ap}  There exists $p>1$ and $C_p > 0$ such that for all $\n \geq 1$ we have $\|h_n\|_{p} \leq C_p$; the $L^{p}$ norm is taken again with respect to Lebesgue.
\end{enumerate}

We have the following 
\begin{proposition}\label{prop::stochasticstability}
    For the random system in Equation \eqref{eq::randomsystem}, under Assumption \ref{itm:Ap}, the sequence of stationary measure $\mu_{\n}$ converges weakly to the unique $T$ invariant probability $\mu$ as $\n\rightarrow\infty$, in the sense that for any real-valued function $g \in C^{0}(I)$, we have
    \begin{equation*}
        \int_{\mathbb{R}} g d\mu_{\n} \rightarrow \int_{\mathbb{R}} g d\mu,\,\,\,\text{as}\,\,\,\n\rightarrow\infty.
    \end{equation*}
\end{proposition}
\begin{proof}
    See Theorem 5.3 in \cite{lillo2021analysis}.
\end{proof}
We will see in the next section that with the preceding assumption we can
prove the convergence of the Lyapunov exponent (Proposition \ref{cly}) and then
verify it numerically on the examples in Section \ref{sec::example}, which is an indirect indication of the validity
of $(A_p).$\\

We conclude this section with the following observations.
\begin{observation}\label{obs::observationone}
Proposition \ref{prop::stochasticstability} is proved by using the representation of the Markov chain in terms of random transformation; in particular, one uses the correlation integral in Equation \eqref{eq::correlationintegral} and the continuity of the map $\eta \rightarrow F_{\eta} \in C^{0}(I)$. 
\end{observation}

\begin{observation}\label{obs::observationtwo}
    Proposition \ref{prop::stochasticstability} can be extended to periodic unimodal maps under the following additional assumption:
    \begin{enumerate}[label=(\text{A}$p$.1)]   
    \item\label{itm:Ap1}  $\forall\,\n$ sufficiently large and $\forall\,x\in\text{supp}(\mu_{\n})$ we have that $|T^{'}(x)|<1$.
    \end{enumerate}
In particular, if $T$ has a globally attracting periodic orbit carrying the discrete measure $\mu$ and satisfies \ref{itm:Ap1}, then the sequence $\mu_{\n}$ converges to $\mu$ in the weak-$^{\star}$topology as $\n\rightarrow\infty$. This requirement can be strengthened by adding the following assumption
    \begin{enumerate}[label=(\text{A}$p$.2)]  
    \item\label{itm:Ap2}   If $T$ is a unimodal periodic map (see Example \ref{ex::unimodalmaps}) and the critical point of the map $\mathfrak{c}$ does not belong to the attracting periodic orbit, then $h_{\n} \rightarrow 0$ uniformly in a neighbourhood of $\mathfrak{c}$ as $\n\rightarrow\infty$.
    \end{enumerate}
\end{observation}

\subsection{Lyapunov Exponent}\label{subsec::lyapunovexponent} 
As done in Subsection 4.3 of \cite{lillo2021analysis}, we define the so-called \textit{average Lyapunov exponent}; see \cite{galatolo2020existence, nisoli2020sufficient}. If the chain admits a unique stationary measure $\mu_{\n}$, then the average Lyapunov exponent is defined as:
\begin{equation}\label{eq::averagelyapunov}
    \Lambda(\mu_{\n}):=\int_{I}\log|T^{'}(x)|d\mu_{\n}.
\end{equation}
In particular, because the stationary measure $\mu_{\n}$ has density of bounded variation, it is enough that $\log |T^{'}| \in L^{p}(\mu_{\n})$ for some $p \geq 1$. For instance, this is the case when $T$ is chaotic or periodic unimodal map (see Example \ref{ex::unimodalmaps}) with a non-flat critical point\footnote{A unimodal map $T$ is said to have a non-flat critical point $\mathfrak{c}$ of order $\ell$ if there is a constant $D$ such that $D^{-1}|x-\mathfrak{c}|^{\ell-1} \leq |T^{'}(x)| \leq D |x-\mathfrak{c}|^{\ell-1}$. In this case, \cite{nowicki1991invariant} prove that the invariant density for $T$ is in $L^{q}$ with $q<\frac{\ell}{\ell-1}$.}.\\
\indent The average Lyapunov exponent in Equation \eqref{eq::averagelyapunov} is introduced to prove that it converges to the analogous quantity computed with respect to the invariant measure $\mu$ of $T$. The following proposition holds.
\begin{proposition}\label{cly}
    Suppose that one of the following conditions is satisfied:
\begin{itemize}
    \item[(a)] The random system in Equation \eqref{eq::randomsystem} verifies \ref{itm:Ap} with the additional assumption that $\log |T^{'}| \in L^{p}(\mu_n)$ for some $p \geq 1$, where $\mu_{\n}$ is the unique stationary measure of the associated Markov chain.  
    \item[(b)] The deterministic map $T$ is a unimodal periodic map (see Example \ref{ex::unimodalmaps}) and verifies Assumptions \ref{itm:Ap1} and \ref{itm:Ap2}. 
\end{itemize}
Then, the average Lyapunov exponent in Equation \eqref{eq::averagelyapunov} converges to the Lyapunov exponent of the deterministic map $T$ as $\n \rightarrow \infty$. Moreover, for $\n$ large enough, it is positive if $T$ verifies \ref{itm:A1}, and negative if $T$ is a periodic unimodal map (see Example \ref{ex::unimodalmaps}, \textit{(i)}).
\end{proposition}
\begin{proof}
    See \cite{lillo2021analysis}, Appendix B, Subsection B.5.
\end{proof}

The average Lyapunov exponent was associated with the phenomenon of \textit{noise induced order} \cite{GN}, which happens when the perturbed system admits a unique stationary measure depending on some parameter, say $\theta$, and the Lyapunov exponent depends and exhibits a transition from positive to negative values. Denote by $\Theta\overset{\text{def}}{=}\{\theta \in \tilde{\Theta}\,|\,\tilde{\Theta}\,\text{is open and}\,\max T_{\theta} < 1\}$ the (extended) parameter space of the map $T$. We use the term ``extended" because also the parameter $\n$ belongs to $\Theta$. Moreover, let index the map $T$ as $T_{\theta}$ to make explicit the dependence on the parameters. Suppose that $T_{\theta}(x) \in C^{3}(\tilde{\Theta} \times I)$ and $p_{\theta}(x,y) \in C^{2}(\tilde{\Theta} \times I^2)$, and let $\bar{\Theta} \subset \tilde{\Theta}$ be the set of parameters for which there exists a unique stationary measure $\mu_{\n}$ with a density of bounded variation. We can now state the following
\begin{theorem}\label{thm::continuityLyapunov}
    The mapping $\bar{\Theta}\ni\theta\mapsto \Lambda_{\theta}\in\mathbb{R}$ is continuous.
\end{theorem}
\begin{proof}
    See \cite{lillo2021analysis}, Theorem 4.12.
\end{proof}

\subsection{Limit theorems}\label{subsec::limittheorems}
We will take advantage of the Markov chain description of our model to state a few important limit theorems for fixed $\n$. These limit theorems are relatively easy to obtain for a fixed $\n$, but they could become very technical for the unperturbed map $T$ because they depend in a non-obvious way on the parameters defining $T$; see, e.g., \cite{young1992decay,thunberg2001periodicity} for a discussion in the case of unimodal maps.\\
\indent As observed above, if $T$ satisfies Assumption \ref{itm:A1}, then the Markov operator $\mathcal{L}$ associated with the Markov chain is quasi-compact on the Banach space $BV$ of bounded variation functions. The adjoint operator $\mathcal{U}$ of $\mathcal{L}$ acts in the following way
\begin{equation*}
    \int_{\mathbb{R}} f_1(x)(\mathcal{L} f_2)(x)\,dx = \int_{\mathbb{R}} (\mathcal{U}f_1)(x)f_2(x)\,dx,
\end{equation*}
where $f_1 \in L^{\infty}$ and $f_2\in L^{1}$. In particular 
\begin{equation*}
    (\mathcal{U}f_1)(x) = \int_{\mathbb{R}} f_1(T_{\eta}(x))d\theta(\eta),
\end{equation*}
where $\theta(\eta)$ is as in Section \ref{sec::randomtrasformation}. In particular, we can write the correlation integral in Equation (\ref{eq::correlationintegral}) in terms of the adjoint operator:
\begin{equation}\label{eq::adjointcorrelation}
    \int_{\mathbb{R}}(\mathcal{L}^{t}h)(x)g(x)\,dx=\int_{\mathbb{R}}\int_{\mathbb{R}}h(x)g(T_{\eta_1}\circ\ldots\circ T_{\eta_t})(x)d\theta^{t}(\bar{\eta}_t)\,dx = \int_{\mathbb{R}}(\mathcal{U}^{t}g)(x)h(x)\,dx
\end{equation}
We take now a function $g \in BV$ such that $\int_{\mathbb{R}} g d\mu_{\n} = 0$. In addition, let $W_k(\bar{\eta}_k, x) = g(T_{\eta_k} \circ\ldots\circ T_{\eta_1})(x)$, where $\bar{\eta}_k=(\eta_1,\ldots,\eta_k)$, and
\begin{equation}\label{eq::sumslimittheorem}
    S_t = \sum_{k=0}^{t-1} W_k.
\end{equation}
We now apply the Nagaeev-Guivarc's perturbative approach \cite{nagaev1961more, guivarc1988theoremes}. This technique enables us to get our limit theorem by \textit{twisting} the transfer operator $\mathcal{L}$; see \cite{hennion2001limit}. Before stating the results, we precise that the underlying probability is $\widetilde{\mathbb{P}}_{\n}\overset{\text{def}}{=}\theta^{\otimes\mathbb{N}}\otimes\mu_{\n}$. If we use this probability, then we should choose a realization $(\eta_t)_{t\geq 1}$ where any $\eta_t \overset{d}{\sim}\eta$, and the initial condition $x \in I$ is chosen $\mu_{\n}$-a.s. The following theorem holds:
\begin{theorem}\label{thm::centrallimittheorem}
Suppose the deterministic map $T$ satisfies Assumption \ref{itm:A1} and $g \in BV$. In addition, let $T_{\eta}$ be the random transformation in Equation (\ref{eq::randomtransformations}). Then, we have:
\begin{enumerate}[label=(e\arabic*)]   
    \item\label{itm:e1} The limit $\iota^{2} \overset{def}{=} \lim_{t\rightarrow\infty} \frac{1}{t}\mathbb{E}_{\widetilde{\mathbb{P}}_{\n}}(S^2_t)$ exists and is equal to
    \begin{equation}
        \iota^2 = \int_{I} g^2\,d\mu_{\n} + 2 \sum_{t=1}^{\infty} g (\mathcal{U}^{t} g)\,d\mu_{\n}.
    \end{equation}
    \item\label{itm:e2} (Central Limit Theorem). Suppose $\iota > 0$. The process $\left(\frac{S_t}{\sqrt{t}}\right)_{t \geq 1}$ converges in law to $\mathcal{N}(0,\iota^2)$ under the probability $\widetilde{\mathbb{P}}_{\n}$. 
    \item\label{itm:e3} (Large Deviation Principle). There exists a non-negative rate function $\mathcal{R}$, continuous, strictly convex, vanishing only at 0, such that for every $\varepsilon$ sufficiently small we have
    \begin{equation*}
        \lim_{t\rightarrow\infty} \frac{1}{t}\log \widetilde{\mathbb{P}}_{\n}(S_t > t \varepsilon) = - \mathcal{R}(\varepsilon).
    \end{equation*}
    \item\label{itm:e4} (Berry-Ess\'en inequality). There exists $D>0$ such that
    \begin{equation}
        \sup_{r \in \mathbb{R}}\Big| \widetilde{\mathbb{P}}_{\n}\left(\frac{S_t}{\sqrt{t}} \leq r\right)-\frac{1}{\iota \sqrt{2\pi}}\int_{-\infty}^{r}e^{-\frac{u^2}{2\iota^2}}\,du\Big| \leq \frac{D\|h_{\n}\|_{BV}}{\sqrt{t}}
    \end{equation}
\end{enumerate}
\end{theorem}
\begin{proof}
    See \cite{aimino2015annealed}, Section 3. 
\end{proof}
We conclude this section with the following
\begin{remark}\label{rmk::remark1}
    The previous theorem hinges on the following exponential decay of correlations (see Proposition \ref{prop::uniqueness}), which is a consequence of the spectral gap prescribed by the Markov operator's quasi-compactness and the uniqueness and mixing property of the absolutely continuous stationary measure. For any observables, $f \in L^{1}$ and $g \in BV$, there exists $0<v<1$ and $C>0$, depending only on the system, such that, for all $k \geq 0$, 
    \begin{equation*}
    \Big|\int_{\mathbb{R}}\int_{\mathbb{R}}f(x)g(T_{\eta_t} \circ\ldots\circ T_{\eta_1})(x)\,d\bar{\eta}\,dx -\int_{\mathbb{R}}f\,d\mu_{\n}\int_{\mathbb{R}}g(x)\,dx\Big|\leq C r^{t} \|f\|_{1}\|g\|_{BV}.
        \end{equation*}
\end{remark}
\subsection{A multifractal analysis}\label{subsec::multifractal1}
We now focus on unimodal maps $T$ of chaotic type, as defined in the Example \ref{ex::unimodalmaps}, and preserving a unique absolute continuous invariant measure $\mu$. The latter is not essentially bounded, but its density is in $L^{p}(\mu)$, for some $p\geq 1$. The presence of divergent values for the density could generate a non-trivial multi-fractal spectrum for the measure $\mu$.\\  
\indent We start with a few reminders about multi-fractal theory; see, e.g., \cite{grassberger1983generalized, jensen1985global, pesin1997multifractal, kantz2004nonlinear, barreira2008dimension}. Let $\mu$ be a probability measure, and $B(x,r)$ the ball of center and radius $r$ on the interval $I$. We denote by 
\begin{equation*}
    d_{\mu}(x) \overset{def}{=} \lim_{r\rightarrow 0}\frac{\log \mu(B(x,r))}{\log r},
\end{equation*}  
the local dimension of the measure $\mu$ at the point $x$, provided that the limit exists. Then, the generalized dimension $D_q(\mu)$, or simply $D_q$, where $q \in \mathbb{Z}$ is obtained as
\begin{equation}\label{eq::generalizeddimension}
    \tau(q) \overset{\text{def}}{=} D_q(q-a) = \inf_{\alpha}\{q\alpha-f(\alpha)\},
\end{equation}
where $f(\alpha)$ denotes the Haursdorff dimension of the set of points for which $d_{\mu}(x)=\alpha$. The previous quantity, also called Legendre transformation, can be linked to the scaling exponent of a suitable correlation integral. In fact, for several dynamical systems $(M, \mu, T)$, where $M$ is a metric space, we have that the following limit
\begin{equation}\label{eq::limit}
    \lim_{r \rightarrow 0}\frac{1}{\log r} \log \int_{M}\mu(B(x,r))^{q-1}\,d\mu
\end{equation}
exists and coincides with $\tau(q)$ in Equation \eqref{eq::generalizeddimension}. Notice that for $q=1$, the limit in Equation \eqref{eq::limit} is replaced by 
\begin{equation}\label{eq::limit2}
       \lim_{r \rightarrow 0}\frac{1}{\log r} \int_{M}\log\mu(B(x,r))\,d\mu,
\end{equation}\\
\indent by an application of the H\^{o}pital's rule. For unimodal maps of Benedicks-Carleson type\footnote{A unimodal map $T$ is of Benedicks-Carleson type if it is defined on the interval $[-1,1]$, is $C^{4}$ and has negative Schwarzian derivative. In addition, if $\mathfrak{c}=0$ is the critical point and $z_k = T^{k}(0)$, then: (i) T is a \textit{Collet-Eckmann} unimodal map verifying $|(T^{k})^{'}(T(\mathfrak{c}))|>\lambda_{\mathfrak{c}}^{k}$, with $\lambda_{\mathfrak{c}} > 1$ $\forall k > H_0$, where $H_0$ is a constant larger than 1; (ii) $T$ verifies the Benedicks-Carleson property: $\exists 0 < \gamma < \frac{\log \lambda_{\mathfrak{c}}}{14}$ such that $|T^{k}(\mathfrak{c})-\mathfrak{c}|>e^{-\lambda k}$, $\forall k > H_0$.} preserving an absolute continuous invariant measure $\mu$, it is possible to compute the spectrum of generalized dimensions. Authors in \cite{baladi2012linear} prove the remarkable result that the density $h$ of $\mu$ has the form
\begin{equation*}
h(x) = \psi_0(x) + \sum_{k\ge1} \frac{\phi_k(x)\chi_k(x)}{\sqrt{|x-z_k|}},
\end{equation*}
with $\psi_0\in C^1$, $\phi_k\in C^1$ is such that $||\phi_k||_{\infty} \le e^{-a k}$ for some $a>0$ and $\chi_k = 1_{[-1, z_k]}$ if $f^k$ has a local maximum at $z_0$, while $\chi_k = 1_{[z_k,1]}$ if $f^k$ has a
local minimum at $z_0$. For such a measure, one can explicitly compute the generalized dimensions via the definition in Equation \eqref{eq::generalizeddimension} (see \cite{caby2022topological}):
\begin{equation}\label{eq::generalizedexplicit}
    D_q = \begin{cases}
    1\,\,\,\,\,\,\,\,\,\,\,\,\,\text{if}\,\,q<2,\\
    \frac{q}{2(q-1)}\,\,\text{otherwise}.
    \end{cases}
\end{equation}

We think that a \textit{similar} result holds for the class of unimodal maps considered in Example \ref{ex::unimodalmaps}. We said similar and not the same result because the non-constant part of $D_q$ depends on the order of divergence at the singular points $z_k$ of the density, which for the Benedicks-Carleson type maps, behaves like $|x-z_k|^{-1/2}$. The values of $D_q$ are constant for negative $q$ whenever the invariant density $h$ is bounded away from zero, see \cite{caby2022topological}; in this case it is also  very ease to see that all the dimensions are less or equal to $1.$ \\
\indent Now, it becomes interesting to explore the spectrum of the generalized dimensions for randomly perturbed orbits. We do not expect any multifractal structure for the stationary measure when its density is essentially bounded, so $D_q=1$, $q \in \mathbb{R}$. Nevertheless the density could become locally very large when $\n\rightarrow \infty$ making it numerically indistinguishable from the unbounded density of the deterministic map on the orbit of the critical point.
To study the dimensions for the stationary measure it is convenient to adopt the point of view of random
transformations (see Section \ref{sec::randomtrasformation}), and consider a \textit{realization} $T_{\eta_t} \circ\ldots\circ T_{\eta_1}$ of a random orbit producing the following empirical measure for a given $\n$:
\begin{equation}\label{eq::empiricalmeasure}
    \nu_{\n, t} = \frac{1}{t}\sum_{j=1}^{t}\delta_{\tilde{\eta}_t},
\end{equation}
where $\tilde{\eta}_t = T_{\eta_{t-1}}\circ\ldots\circ T_{\eta_{1}}(x)$ for a suitable point $x$ (see below). Again, each $\eta_k$ has distribution $\theta$. From the ergodic theorem for random transformations, it now follows that 
\begin{equation}
    \int_{\mathbb{R}} g d\,\nu_{\n,t} = \frac{1}{t}\sum_{j=1}^{t} g(\tilde{\eta}_t)\rightarrow\int_{\mathbb{R}} g\,d\mu_{\n}\quad\text{as}\quad t\rightarrow\infty,
\end{equation}
where $\mu_{\n}$ is the stationary measure, $g \in L^{1}(\mu_{\n})$, and the point $x$ is chosen $\mu_{\n}$-a.e., and the sequence $(\eta_t)_{t\geq 1}$ is chosen $\theta^{\otimes \mathbb{N}}$-a.e.. Since the support of $\mu_{\n}$ contains the dynamical core, by taking an arbitrary point $x$ in such a core and by fixing a realization $(\eta_t)_{t\geq 1}$, the generalized dimensions of the stationary measure $\mu_{\n}$ could be computed directly via the correlation integral formula \eqref{eq::limit}by using the empirical measure \eqref{eq::empiricalmeasure} for large $t$; see, e.g., \cite{caby2022topological}.
\section{Extreme values distribution}\label{sec::extremevalues}
In this subsection, we develop an EVT for the Markov chain defined in \ref{sec::markovchain} for finite values for the parameter $\n$. In particular, we consider the chain $(X_t^{(\n)})_{t \geq 1}$ with the stochastic kernel $p_{\n}(x,y)$, endowed with the canonical probability $\mathbb{P}_{\n}$ having initial distribution $\mu_{\n}=h_{\n}\,dx$. We focus on the derivation of the Gumbel law for a particular observable by deriving the distribution of the first entrance of the chain in a small set, which we name \textit{rare set}\footnote{See, e.g., the monograph \cite{leadbetter2012extremes} for a general presentation of EVT.}. To this aim, we index with $t$ the rare set defined as a ball of center $z \in I$ and with radius $e^{-u_t}$, $B_{t}(z)\overset{\text{def}}{=}B(z,e^{-u_t})$, where $u_t$ is a sequence called \textit{boundary levels} such that $u_t\rightarrow\infty$ as $t\rightarrow\infty$, and verifying  
\begin{equation}\label{eq::limitevt}
    t \mu_{\n}(B(z,e^{-u_t}))\rightarrow\tau\quad\text{as}\quad t\rightarrow\infty,
\end{equation}
where $\tau\in\mathbb{R}_{>0}$. Then, we consider the observable
\begin{equation}\label{eq::observables}
    \varphi(x)\overset{\text{def}}{=}-\log(\text{dist}(x,z)),
\end{equation}
where $x \in I$, and $\text{dist}(\,\cdot\,)$ denotes the usual distance on $\mathbb{R}$. Then, we define the following random variable with values in $I$
\begin{equation}\label{eq::randomvariableEVT}
    M_t^{(\n)}\overset{\text{def}}{=}\max\{\varphi \circ X_0^{(\n)},\ldots,\varphi\circ X_{t-1}^{(\n)}\}.
\end{equation}
We will be interested in the distribution $\mathbb{P}_{\n}(M_t^{(\n)}\leq u_t)$ as $t\rightarrow\infty$. In particular, by the stationarity of the Markov chain, this distribution is equivalent
to the probability that the first entrance of the chain into the ball $B_{t}(z)$ is larger than $t$.\\
\indent Condition \eqref{eq::limitevt} enables us to get verifiable prescriptions on the sequence of boundary levels $u_t$. If the stationary measure $\mu_{\n}$ is non-atomic, then the measure of a ball is a continuous function of the radius. Therefore, for any given $\tau \in \mathbb{R}_{+}$ and $t\in\mathbb{N}_{\geq 1}$, we can find $u_t$ such that $\mu_{\n}(B(z,e^{-u_t}))=\frac{\tau}{t}$. Now, we denote by $B_t^{c}(z)$ the complement of the ball $B_t(z)$, and define the perturbed operator $\tilde{\mathcal{L}}_{(t)}$ for $g \in BV$ as
\begin{equation}\label{eq::perturbedoperator}
    \tilde{\mathcal{L}}_{(t)} g \overset{\text{def}}{=}\mathcal{L}(g 1_{B_t^{c}(z)}).
\end{equation}
It is straightforward to check that
\begin{equation}\label{eq::probabilityEVT}
    \begin{split}
        &\mathbb{P}_{\n}(M_t^{(\n)}\leq u_t) = \mathbb{P}_{\n}(X_0^{(\n)} \in B_t^{c}(z), \ldots, X_{t-1}^{(\n)} \in B_t^{c}(z))\\
        &=\int_{B_t^{c}(z)}h_{\n}\,dx_0\int_{B_t^{c}(z)}p_{\n}(x_0,x_1)\,dx_1\ldots\int_{B_t^{c}(z)}p_{\n}(x_{t-1},x_
        {t})\,dx_t\\
        &= \int_{\mathbb{R}}(\tilde{\mathcal{L}}_{(t)}h_{\n})(x)\,dx.
    \end{split}
\end{equation}
We now show that the operator $\tilde{\mathcal{L}}_{(t)}$ approaches $L$ in a precise sense that allows us to control the asymptotic behavior of the integral in \eqref{eq::probabilityEVT}. This result allows us to control the asymptotic behavior  of the integral in \eqref{eq::probabilityEVT}. In order to make the argument rigorous, we need more assumptions on the operator $\mathcal{L}$, in addition to the quasi compactness. The same quasi compactness property is shared by the operator $\tilde{\mathcal{L}}_{(t)}$, provided that $t$ is large enough, and provided that $\tilde{\mathcal{L}}_{(t)}$ is close to $\mathcal{L}$ in the following sense
\begin{equation}\label{eq::triplenorm}
    \|(\mathcal{L}-\tilde{\mathcal{L}}_{(t)})(g)\|_{1} \leq c(t)\|g\|_{BV},
\end{equation}
where $c(t)\rightarrow\infty$ as $t\rightarrow\infty$. Indeed, we have
\begin{equation}\label{eq::lconvergence}
    \int_{\mathbb{R}}|(\mathcal{L}-\tilde{\mathcal{L}}_{(t)})(g)|\,dx = \int_{\mathbb{R}}|\mathcal{L}(g 1_{B_t})|\,dx \leq \int_{\mathbb{R}} \mathcal{L}(|g|1_{B_t})\,dx \leq \|g\|_{BV}\text{Leb}(B_t),
\end{equation}
because the space $BV$ is continuously embedded into $L^{\infty}$ with constant equal to one. We can apply the perturbation theorem of Keller-Liverani, which gives the asymptotic behavior of the top eigenvalue of $\tilde{\mathcal{L}}_{(n)}$ around one; see \cite{keller2009rare, keller2012rare}. In addition, see, e.g., \cite{lucarini2016extremes}, Chapter 7, for an application of that theory to Markov chains. At this point, we need a further assumption:
\begin{enumerate}[label=(E\arabic*)]   
    \item\label{itm:E1} The density $h_{\n}$ of the stationary measure is bounded away from zero on the rare set $B_{t}(z)$.
\end{enumerate}
Therefore, we can prove that 
\begin{equation*}
    \mathbb{P}_{\n}(M_t^{(\n)}\leq u_t)\rightarrow e^{-\theta \tau},\quad\text{as}\quad t\rightarrow\infty,
\end{equation*}
where the so-called \textit{extremal index} (EI) $\theta$ satisfies
\begin{equation}\label{eq::extremalindex}
   \theta = 1-\sum_{k=0}^{\infty}q_k,
\end{equation}
with $q_{k}=\lim_{t\rightarrow\infty} q_{k,t}$, provided that the limit exists, with:
\begin{equation}\label{eq::extremalindex2}
   q_{k,t} = \frac{\mathbb{P}_{\n}(X_0^{(\n)}\in B_t(z), \ldots, X_k^{(\n)}\in B_t(z), X_{k+1}^{(\n)}\in B_t(z))}{\mu_{\n}(B_t(z))}.
\end{equation}
Namely, $q_{k,t}$ is the probability of $\mu_{\n}$-distributed stationary chain to start in $B_t(z)$ and then return to it after exactly $(k+1)$ steps. It is now easy to show that all the $q_{k,t}$ vanishes in the limit as $t\rightarrow\infty$ since we have
\begin{equation}\label{eq::extremalindex3}
   q_{k,t} \leq \frac{\mathbb{P}_{\n}(X_0^{(\n)}\in B_t(z), X_{k+1}^{(\n)}\in B_t(z))}{\mu_{\n}(B_t(z))} \leq \frac{c_{\n}\text{Leb}(B_t(z))\mu_{\n}(B_t(z))}{\mu_{\n}(B_t(z))},
\end{equation}
where, to estimate the right-hand side of \eqref{eq::extremalindex3}, we use the fact that for a fixed $\n$, the stochastic kernel $p_{\n}(x,y)$ is uniformly bounded by a constant $c_{\n}$. In particular, the right-hand side converges to zero as $t\rightarrow\infty$. We have just proved the following
\begin{proposition}\label{prop:extremevalue}
Suppose that our Markov chain is constructed upon a  map $T$  verifying Assumption \ref{itm:A1}, and that $\mathbb{P}_{\n}$ is the canonical probability with initial distribution $\mu_{\n}=h_{\n} dx$. Then, we get Gumbel's law:
$$
\lim_{t\to\infty}\mathbb{P}_{\n}(M_t^{(\n)}\le u_t)=e^{-\tau},
$$
where $M_t^{(\n)}$ is defined in Equation \eqref{eq::randomvariableEVT}, $\varphi(\cdot)$ in Equation \eqref{eq::observables}, the boundary level $u_t$ verifies $\mu_{\n}(B(z,e^{-u_t}))=\frac{\tau}{t}$, and on the set $B(z,e^{-u_t})$ the density $h_{\n}$ of the stationary measure is bounded away from zero for large $t$ (Assumption \ref{itm:E1}).
\end{proposition}
Our Markov chain visits infinitely often the neighborhood $B_t(z)$ of any point $z$. Therefore, we expect that the exponential law $e^{-\tau}$ given by the extreme value distribution describes the time between successive events in a Poisson process. To formalize this,
we introduce the random variable
\begin{equation}\label{eq::randomvariableevt}
    \mathcal{N}_{z}^{(t)}(s):=\sum_{k=0}^{\lfloor\frac{s}{\mu_{\n}(B_t(z))}\rfloor} 1_{B_t(z)}(X_k^{(\n)}), 
\end{equation}
and we consider the following distribution
\begin{equation}\label{eq::distributionevt}
    \mathbb{P}_{\n}(\mathcal{N}_{z}^{(t)}(s) = k).
\end{equation}
We have the following
\begin{proposition}\label{prop:extremevalue1}
Suppose that our Markov chain is constructed upon a  map $T$  verifying Assumption \ref{itm:A1}, and that $\mathbb{P}_{\n}$ is the canonical probability with initial distribution $\mu_{\n}=h_{\n} dx$. Then, we have:
\begin{equation}
 \lim_{t\to\infty}\mathbb{P}_{\n}(\mathcal{N}_{z}^{(t)}(s)=k)=\frac{t^{k} e^{-s}}{k!},   
\end{equation}
where the density $h_{\n}$ of the stationary measure is bounded away from zero for large $t$ on  the set $B(z,e^{-u_t})$ (Assumption \ref{itm:E1}).
\end{proposition}
\begin{proof}
    See \cite{nicolai2019limiting}. 
\end{proof}

In particular, we have shown that the EI is equal to 1. Such an index is less than one when clusters of successive recurrences happen, which is the case, for instance, when the target point $z$ is periodic. Our heteroscedastic noise breaks periodicity, so we expect an EI equal to one.

We conclude this section with the following observation and example. 
\begin{observation}
While we rigorously prove an EVT for the Markov chain, we are still determining if a similar result holds for the deterministic map $T$ with respect to its invariant measure. Moreover, there are, in fact, only a few results on EVT for unimodal maps; see, for instance, \cite{freitas2008link, collet2001statistics,moreira2010extremal}.
\end{observation}

\section{An application to systemic risk}\label{sec::example}
This section presents a stylized model of the leverage dynamics to which our theory applies. A part from providing a potential application of the models considered in this paper, we will use the specific model to perform numerical simulations of the maps and to test the finite size effect of some asymptotic results presented above. The model is an extension of the one presented in  \cite{lillo2021analysis} since we add here a possible relation between liquidity and leverage, whereas in \cite{lillo2021analysis} liquidity was considered constant. The description of the model follows the same lines as the presentation in \cite{lillo2021analysis}.\\
\indent A representative financial institution (hereafter a bank) takes investment decisions at discrete times $t \in \mathbb{Z}$, which defines the slow time scale of the model. At each time the bank's balance sheet is characterized by the asset $A_t$ and equity $E_t$, which together define the leverage $\lambda_t := A_t/E_t$. The bank wants to maximize leverage (by taking more debt) to increase profits, but regulation constraints the bank's Value-at-Risk (VaR) in such a way that $\lambda_t = \frac{1}{\alpha \sigma_{e,t}}$, where $\alpha$ depends on the return distribution and VaR constraint\footnote{For example, if returns are Gaussian and the probability of VaR is $5\%$, it is $\alpha=1.64$.}, and $\sigma_{e,t}$ is the expected volatility at time $t$ of the asset, which in this model is composed by a representative risky investment. Thus at each time $t$ the bank recomputes $\sigma_{e,t}$ and chooses $\lambda_t$.
Then, in the interval $[t,t+1]$ the bank trades the risky investment to keep the leverage close to the target $\lambda_t$.
The trading process occurs on the points of a grid obtained by subdividing $[t,t+1]$ in $\n$ subintervals of length $1/\n$ (the fast time scale).
The dynamics of the investment return can be written as
\begin{equation}\label{eq:riskyreturn}
    r_{t+k/\n} = \varepsilon_{t+k/\n} + e_{t+(k-1)/\n},\quad k=1, 2, \ldots, \n,
\end{equation}
where $\varepsilon_{t+k/\n}$ and $e_{t+(k-1)/\n}$ are, respectively, the exogenous and endogenous component of the return.
The former is a white noise term with variance $\sigma^2_\epsilon$, while the latter depends on the banks' demand for the risky investment in the previous step.
For each bank, the demand for the risky investment at time $t+k/\n$ is the difference between the target value of $A_t$ to reach $\lambda_t$ and its actual value.
Since the bank's asset is composed by the risky investment, an investment return $r_{t+k/\n}$ modifies $A_t$ and the bank trades at each grid point to reach the target leverage.
It is possible to show (see \cite{corsi2016micro,mazzarisi2019panic}) that to achieve this, at each time $t+k/\n$ the bank's demand for the risky investment is
$$
D_{t+k/\n}=(\lambda_t-1) A^*_{t+(k-1)/\n} r_{t+k/\n}, 
$$
where $A^*_{t+(k-1)}$ is the target asset size in the previous step.
If there are $M$ identical banks, the aggregated demand is $MD_{t+k/\n}$.
The endogenous component of returns $e_{t+k/\n}$ is determined by the aggregated demand by the equation
\begin{equation}\label{eq::equationgamma}
   e_{t+k/\n}=\frac{1}{\gamma_t}\frac{MD_{t+k/\n}}{C_{t+k/\n}},
 \end{equation}
where  $C_{t+k/\n}=MA^*_{t+(k-1)/\n}$ is a proxy of the market capitalization of the risky asset, and
$\gamma_t$ is a parameter measuring at each point in time the investment liquidity. Notice that in \cite{lillo2021analysis} this parameter is considered constant. Using the above expression, it is
$$
e_{t+k/\n}=\frac{\lambda_t-1}{\gamma_t} e_{t+(k-1)/\n} = \phi_t e_{t+(k-1)/\n}
$$
and thus in the period $[t, t+1]$ the return $r_{t+k/\n}$  follows an AR(1) process with autoregression parameter $\phi_t=(\lambda_t-1)/\gamma_t$ and idiosyncratic variance $\sigma^2_\epsilon$. In the present paper, we assume that $\gamma_t$ is linked to the level of the leverage $\lambda_t$ by the following relation:
\begin{equation}\label{eq::affinerelation}
    \gamma_t = \gamma_0 + c \lambda_t,
\end{equation}
where $\gamma_0$ is a positive constant, and $|c|\leq 1$. As far as we know, there is not a
unified consensus in the literature on the type (linear or not), and the sign
of the relationship between the market\footnote{In the financial literature, one finds also the notion of book leverage. Book leverage is defined as the ratio of total assets to book equity, while market leverage is defined as the ratio of enterprise value (total assets - book equity + market equity) to market equity. Empirically, book-measured leverage and market-measured leverage lead to different inferences about the time series properties of leverage; see the debate between \cite{adrian2014financial} and \cite{he2017intermediary}. We here refer to the market leverage in our discussion because in order to be consistent with our empirical application in \cite{lillo2021analysis}.} leverage and liquidity. For instance, \cite{wojcik2015determinants} states, ``The relationship between market leverage ratio and liquidity risk in the long term is negative and statistically significant \textit{only} for commercial banks belonging to the old EU countries". In particular, it seems that there is no a universal statement on the sign of $c$. As regards as the type of dependence, we decide for a linear relationship. Admittedly, the linear relationship may seem too crude, but a non linear dependence would be an additional technicality that would not add to the present work’s conceptual advancements.\\
\indent To close the model, we specify how the bank forms expectations $\sigma_{e,t}$ on future volatility at time $t$.
We assume that bank uses adaptive expectations, which implies that
$$
\sigma^2_{e,t}=\omega \sigma^2_{e,t-1}+(1-\omega)\hat \sigma^2_{e,t},
$$
where $\omega \in [0,1]$ is a parameter weighting between the expectation at $t-1$ and the estimation $\hat\sigma^2_{e,t}$ of volatility   obtained by the return data in $[t-1,t]$.
As done in practice, this is obtained by estimating the sample variance of the returns in $[t-1,t]$, i.e.
\begin{multline}
\hat\sigma^2_{e,t} = \widehat{\text{Var}}\left[\sum_{k=1}^{\n} r_{t-1+k/\n}\right] \\
= \left(1+2\frac{\hat \phi_{t-1}(1-\hat \phi_{t-1}^\n)}{1-\hat \phi_{t-1}}-2\frac{(\n\hat \phi_{t-1}-\n-1)\hat \phi_{t-1}^{\n+1}+\hat \phi_{t-1}}{\n(1-\hat \phi_{t-1})^2}\right) \frac{\n \hat \sigma_{\epsilon}^2}{1-\hat \phi^2_{t-1}},
\end{multline}
where the last expression gives the aggregated variance of an AR(1) process as a function of the AR estimated parameters $\hat \phi_{t-1}$ and $\hat \sigma^2_\epsilon$.
In the following we will assume that these are the Maximum Likelihood Estimators (MLE).
We remind that when $\n$ is large, $\hat \phi_{t-1}$ is a Gaussian distributed variable with mean $\phi_{t-1}$ and variance $(1-\phi^2_{t-1})/\n$.

In conclusion, the leverage dynamics is described by the following equations:
\begin{equation}\label{eq:model_final_1}
    \begin{cases}
    & \lambda_t=\left(\omega\frac{1}{\lambda^2_{t-1}}+(1-\omega)\alpha^2 \widehat{\text{Var}}[\sum_{k = 1}^{\n} r_{t-1+k/\n}]\right)^{-1/2},\\
    & r_s = \phi_{t-1}r_{s-1/\n} + \epsilon_s,\qquad s=t-1+k/\n,\quad k=1,2,\ldots,\n,
    \end{cases}
\end{equation}
Since slow variables evolve depending on averages of the fast variables, the model is a {\em slow-fast deterministic-random dynamical system}.
By using the expression above for the variance, we can rewrite the equation for the slow component only as
$$
\lambda_t=\left(\omega\frac{1}{\lambda^2_{t-1}}+(1-\omega)\alpha^2 \hat \sigma^2_{e,t}\right)^{-1/2},
$$
where the estimator $\hat \sigma^2_{e,t}$ can be seen as a stochastic term depending on $\lambda_{t-1}$ and whose variance goes to zero when $\n\to\infty$. 

If $\n$ is large, the above map reduces to
\begin{equation*}
    \lambda_t=\left(\omega\frac{1}{\lambda^2_{t-1}}+\frac{(1-\omega)\alpha^2 \n \hat \sigma_{\epsilon}^2}{(1-\hat \phi_{t-1})^2}\right)^{-1/2},
\end{equation*}

When changing $\n$ also $\sigma^2_\epsilon$ changes, since the AR(1) can be seen as the discretization of a continuous time stochastic process (namely an Ornstein-Uhlenbeck process).
A simple scaling argument shows that the quantity $\Sigma_{\epsilon}=\sigma^2_\epsilon \n$ is instead constant and independent from the discretization step $1/\n$. With abuse of notation, we set: $\Sigma_{\epsilon}\overset{\text{def}}{=}\lim_{\n\rightarrow\infty} \n \hat{\sigma}^2_{\epsilon}$, and we define $\overline{\Sigma}_{\epsilon}\overset{def}{=}(1-\omega)\alpha^2\Sigma_{\epsilon}$. At this point, we observe that since in the large $\n$ limit the MLE estimator $\hat{\phi}_{t-1}$ is a Gaussian variable with mean $\phi_{t-1}$ and variance $(1-\phi^2_{t-1})/\n$, we can write
\begin{equation*}
    \hat{\phi}_{t-1}=\phi_{t-1}+\eta_{t-1},\quad\quad\eta_{t-1}\overset{d}{\sim}\mathcal{N}\left(0,\frac{(1-\phi^2_{t-1})}{\n}\right).
\end{equation*}
By using the definition of $\gamma_t$ in Equation \eqref{eq::affinerelation}, by defining $\phi_t\overset{\text{def}}{=}\frac{\lambda_t-1}{\gamma_t}$, and by introducing the function $V : \mathbb{R}^2\rightarrow\mathbb{R}$ given for any $(u,v)\in\mathbb{R}^2$ by
\begin{equation}\label{eq::functionV}
    V(u,v)\overset{\text{def}}{=}\left(\frac{\omega(1-c u)^2}{(1+\gamma_0 u)^2}+\frac{\overline{\Sigma}_{\epsilon}}{(1-(u+v))^2}\right)^{-1/2},
\end{equation}
we get
\begin{equation}\label{eq::functionphi}
    \phi_t = \frac{V(\phi_{t-1},\eta_{t-1})-1}{\gamma_0 + c V(\phi_{t-1},\eta_{t-1})} \overset{\text{def}}{=}\mathcal{F}(\phi_{t-1},\eta_{t-1})
\end{equation}
\noindent If the noise $\eta_{t-1}$ is small (i.e., $\n$ is large), we can perform a series expansion, obtaining:
\begin{equation*}
    V(\phi_t,\eta_t) = A(\phi_t) + B(\phi_t)\eta_t,
\end{equation*}
where 
\begin{equation*}
    \begin{split}
        A(u)&\overset{\text{def}}{=}\frac{1+\gamma_0 u}{[\omega(1-c u)^2 + \overline{\Sigma}_{\epsilon}(1-u)^{-2}(1+\gamma_0 u)^2]^{1/2}}\\
        B(U)&\overset{\text{def}}{=}\frac{ \overline{\Sigma}_{\epsilon}(1-u)^{-1}(1+\gamma_0 u)^{3}}{[\omega(1-c u)^2 + \overline{\Sigma}_{\epsilon}(1-u)^{-2}(1+\gamma_0 u)^2]^{3/2}}\\
    \end{split}
\end{equation*}
Accordingly, Equation \eqref{eq::functionphi} becomes:
\begin{equation*}
    \phi_t = \frac{A(\phi_{t-1})+\eta_{t-1}B(\phi_{t-1})-1}{\gamma_0 + c A(\phi_{t-1}) + c \eta_{t-1} B(\phi_{t-1})}.
\end{equation*}
By performing, again, a series expansion we obtain:
\begin{equation}\label{eq::final}
    \begin{split}
        \phi_t  &= \frac{A(\phi_{t-1})-1}{\gamma_0 + c A(\phi_{t-1})}+\frac{(1-\phi_{t-1^2})^{1/2}(\gamma_0+c)B(\phi_{t-1})}{\sqrt{\n}(\gamma_0 + c A(\phi_{t-1}))}\tilde{\eta}_{t-1},\\
                &\overset{def}{=} T(\phi_{t-1})+\sigma_{\n}(\phi_{t-1})\tilde{\eta}_{t-1}
    \end{split}
\end{equation}
with $\tilde{\eta}_{t-1}\overset{d}{\sim}\mathcal{N}(0,1)$, $t \in \mathbb{N}_{\geq 1}$. Notice that we performed a series expansion for $\eta_{t-1}$  small, which is justified whenever the variable $\phi$ stays far from one. Because we are going to iterate the map $\mathcal{F}$ in Equation \eqref{eq::functionphi} for $\phi \in [0,1]$ and $|\eta|\ll 1$, it is enough to show that:
\begin{equation*}
    \max_{\phi \in [0,1], |\eta|\ll 1} |\mathcal{F}(\phi,\eta)|,<1.
\end{equation*}
because in this case all the successive iterates $|\mathcal{F}^{t}(\phi,\eta)|$, $t\in\mathbb{N}_{\geq 1}$, satisfy the same bound. It is not difficult to see that the bound holds true provided that $\gamma_0$ is sufficiently large.\\ 
Now, we study the deterministic map $T$ in Equation \eqref{eq::final}, which is the deterministic component  of $\mathcal{F}$ for $\eta$ small.  By arguing as above, we have that
\begin{equation}
    \Delta \overset{\text{def}}{=}\max_{\phi \in [0,1]}|T(\phi)|<1,
\end{equation}
provided that $\gamma_0$ is sufficiently large.
\begin{figure}
    \centering
    \includegraphics[scale=0.55]{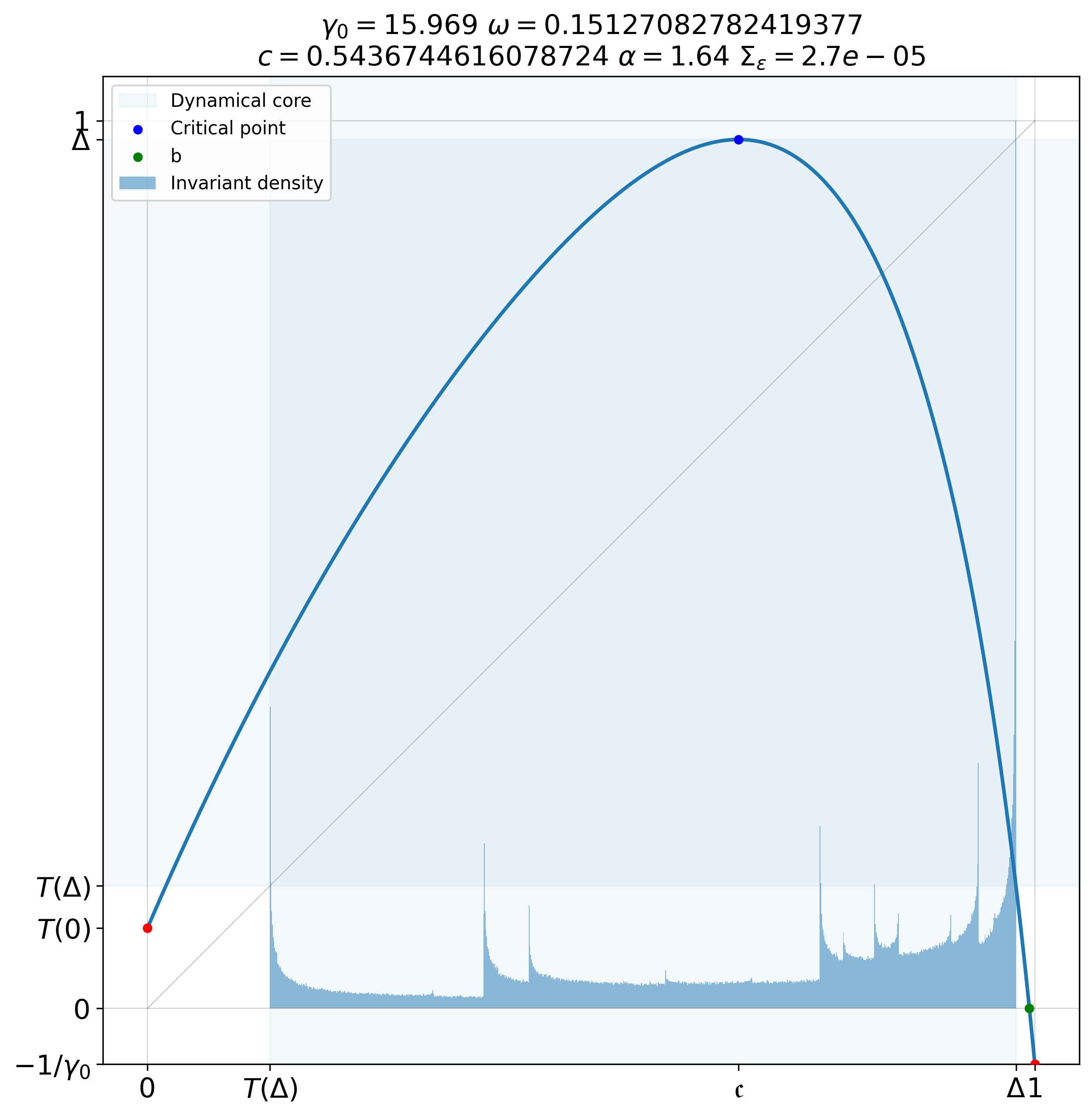}
    \caption{Plot of the deterministic component $T(\phi)$, $\gamma_0=15.969$, $\alpha=1.64$, $\Sigma_{\epsilon}=2.7\times 10^{-5}$. The value for $\gamma_0$ is taken from the empirical analysis in \cite{lillo2021analysis}, Section 7.2, (where it is denoted simply by $\gamma$) . The value $\alpha=1.64$ corresponds to a VaR constraint of $5\%$ in case of a Gaussian distribution for the returns. The values $\Sigma_{\epsilon}=2.7\times10^{-5}$ is taken from \cite{mazzarisi2019panic}, Table 1, and corresponds to the exogenous idiosyncratic volatility at the time scale of portfolio decisions. The value for $\omega$ and $c$ are randomly sampled from the dynamical core, once fixed the other parameters. The \textit{Blue dot} indicates the critical point $\mathfrak{c}$, the \textit{Green dot} the intersection between the map and the horizontal axis, the left-hand \textit{Red dot} indicates the image of 0, the right-hand \textit{Red dot} indicates $\lim_{\phi\rightarrow1^{-}}T(\phi)=-\frac{1}{\gamma_0}$. The support of the invariant density belongs to the so-called \textit{dynamical core} $[T(\Delta),\Delta]$.} 
    \label{fig::deterministicmap}
\end{figure}

Figure \ref{fig::deterministicmap} shows the map $T$ for some suitably chosen parameters:
\begin{itemize}
    \item $\gamma_0=15.969$; this value is taken from the empirical analysis in \cite{lillo2021analysis}, Section 7.2 (where it is denoted simply by $\gamma$). It corresponds to the maximum value of the leverage computed over a 4,389 time series of US Commercial Banks and Saving and Loans Associations; see \cite{lillo2021analysis}, Section 7.1, for a detailed description of the dataset.  
    \item $\alpha=1.64$; it corresponds to a VaR constraint of $5\%$ in case of a Gaussian distribution for the returns.
    \item $\Sigma_{\epsilon}=2.7\times10^{-5}$; this value is taken from the numerical analysis in \cite{mazzarisi2019panic}, Table 1, and corresponds to the exogenous idiosyncratic volatility at the time scale of portfolio decisions.
    \item The values for $\omega$ and $c$ are free parameters and are randomly sampled (e.g., from the dynamical core).
\end{itemize}

The figure shows that that $T$ is a unimodal map with a negative Schwarzian derivative (see below). In the figure, $\Delta$ is the iterate of the unique critical point (\textit{Blue dot}) $\mathfrak{c}$ of $T$, i.e., $\Delta=T(\mathfrak{c})$. Therefore, if we take the initial condition $\phi_0$ in the interval $[\Delta, 1]$, then all the successive iterates $|T^{t}(\phi_0)|$, $t\in\mathbb{N}_{>1}$ will stay in $[0, \Delta]$.\\
\indent By definition, the (re-scaled) leverage of the representative bank is a positive quantity. However, as one can also notice from the graph in Figure \ref{fig::deterministicmap}, we have that $\lim_{\phi \rightarrow 1^{-}} T(\phi) = -\frac{1}{\gamma_0}$. Therefore, we need to slightly modify the definition of our map by restricting it to the interval $[0,b]$, being $b$ the point of intersection between the map and the horizontal axis (\textit{Green dot} in Figure \ref{fig::deterministicmap}) Notice that this definition makes sense when $\Delta < b < 1$. In addition, as we verified numerically, if we take the initial condition in the interval $[\Delta, b]$, then all the other iterates will stay in $[0,\Delta]$. In particular, the previous redefinition is legitimate also if we consider the effect of the noise. Indeed, it is clear from the considerations in Section \ref{sec::assumptions} that if we symmetrize about the horizontal axis, the graph of $T$ in the interval $[b,1]$ to make it positive, then the equilibrium state for the chain, precisely its unique stationary measure, has support that does not intersect the interval $[b,1]$ if $a$ satisfies the bound in Equation \eqref{eq:boundona2}. Also, we verified numerically that the condition $\Delta < b < 1$ holds for a $\gamma_0$ sufficiently large. We continue to denote by $T$ the map after this redefinition. We now modify the map $T$ by enlarging on the left its domain of definition to take into account the action of the additive noise. To do so, we first notice that 
\begin{equation*}
    T(0) = a = \frac{1-\sqrt{\omega+\overline{\Sigma}_{\epsilon}}}{\gamma_0\sqrt{\omega+\overline{\Sigma}_{\epsilon}}+c}>0;
\end{equation*}       
see, the \textit{Red-left dot} in Figure \ref{fig::deterministicmap}. With abuse of notation, (re)define\footnote{Cfr. Equation \eqref{eq::equationgamma}} $\Gamma \overset{\text{def}}{=} b-\Delta$, and extend the domain of definition $T$ to the larger interval $[-\Gamma, b]$ so that $T$ is continuous at $0$ and on $[-\Gamma,0)$ is $C^{4}$ smooth, positive and decreasing, with $T(-\Gamma)<\Delta$. Again, with abuse of notation, we will still denote by $T$ the map after this second redefinition, and, hereafter, write  $I \overset{def}{=}[-\Gamma,b]$.\\
\noindent   The map $T$ just-defined verifies Assumption \ref{itm:B12} and we choose the distribution of the random variables $(\tilde{\eta}_{t})_{t \geq 1}$ in order to satisfy Assumption \ref{itm:C1}. We need to verify Assumption \ref{itm:A1}-(c). In order to do so, we verify numerically the following important result taken from \cite{keller1990exponents} (see, also, \cite{hennion2001limit}, Theorem 12). Define the number
\begin{equation}\label{eq::uniqueness}
    \ell_{T}(x) =\lim_{t \rightarrow \infty}\frac{1}{t}\log|(T^{t})^{'}(x)| = \lim_{t \rightarrow \infty} \sum_{i=0}^{t-1}\log|T^{'}(T^{t}(x))|,\quad x \in I.
\end{equation}
Suppose $T$ is a unimodal map with negative a Schwarzian derivative, and non-flat critical point with $\ell_{T}(x)=\kappa>0$ for \textrm{Leb}-almost all $x \in I$, then $T$ admits a unique absolutely continuous invariant probability measure $\nu$. In this case, $\kappa$ will be the Lyapunov exponent of the map $T$ with respect to $\nu$. Figure \ref{fig::indicatorabs} represents the value of $\ell_{T}$ in the same parameters configuration of Figure \ref{fig::deterministicmap}.\\
\begin{figure}
    \centering
    \includegraphics[scale=0.55]{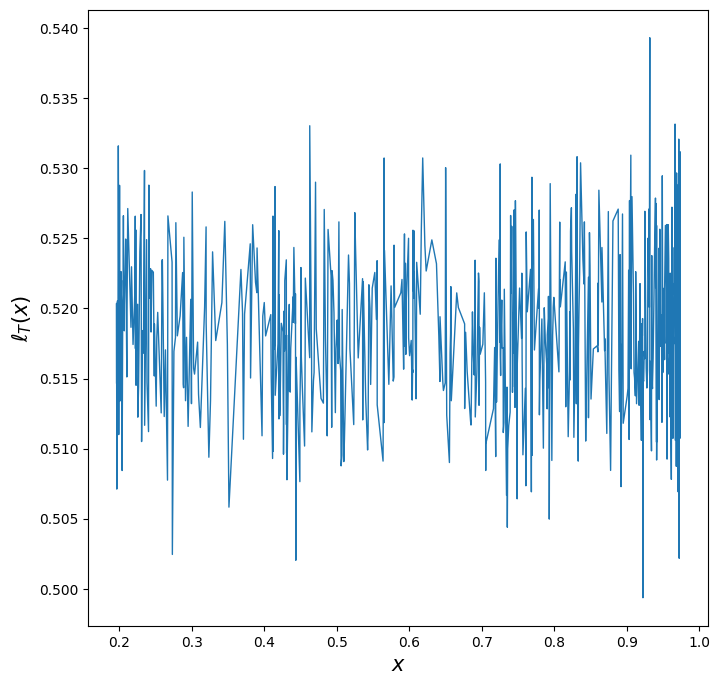}
    \caption{Plot of the indicator in Equation \eqref{eq::uniqueness}. \textit{Parameters' configuration:} See the caption of Figure \ref{fig::deterministicmap}.} 
    \label{fig::indicatorabs}
\end{figure}

\indent Once we have verified that our systemic risk model in Equation \eqref{eq::final} satisfies Assumptions \ref{itm:A1}, \ref{itm:B1}, and \ref{itm:C1}, we pass to investigate whether it satisfies, as it should be, the mathematical properties in Section \ref{sec::mathematical_properties} and the EVT in Section \ref{sec::extremevalues}. The order in which we present the results reflects the order in which they were presented in the latter sections.

\subsection{Dynamics properties of the map}\label{subsec::stationarymeasureandstoch}
The bifurcation diagram of a dynamical system shows how the asymptotic distribution of a typical orbit varies as a function of a parameter. For our map, either the memory parameter $\omega$ or the parameter $c$ can be employed as bifurcation parameter. Figure \ref{fig::bifurcation} shows the bifurcation diagram as a function of $c \in [-1,1]$. The choice of the parameter $\omega$ for this plot corresponds to a value of $\omega$ for which a specific pair $(c,\omega)$ is in the dynamical core ($\omega=0.669$).\\
\indent We now comment Figure \ref{fig::bifurcation}. Moving backward, between 1 and 0.3 there is an attracting fixed point. Then, as $c$ gets smaller and smaller, the period one behaviour splits into period two and the two values are getting further apart. The situation is more complex for $c$ in $[-0.48,0.3]$ as small parameter variations can change the dynamics from chaotic to periodic and back. Finally, when $c$ is between $-0.48$ and $-1$ there is an attracting fixed point. However, in this range, $\phi_t$ takes negative values and this does not make sense in our financial application, since it would correspond to negative leverage. Figure \ref{fig::framebyframebifurcation} shows how the graph of the map $T$ changes as a function of $c$ reflecting the description of the bifurcation diagram.
  
\begin{figure}
    \centering
    \includegraphics[scale=0.155]{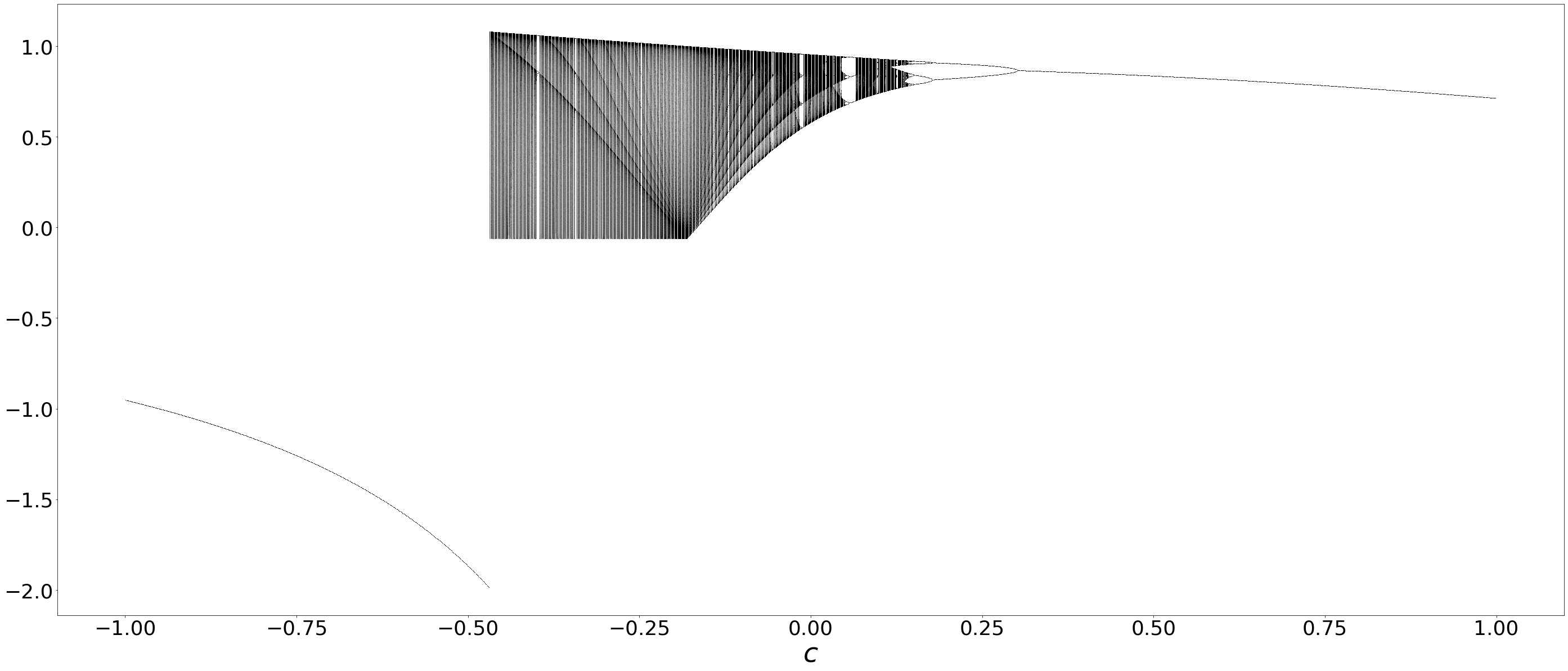}
    \caption{Bifurcation diagram for $T$. \textit{Parameters' configuration:} $(\gamma_0, \alpha, \Sigma_{\epsilon}, \omega)=(15.969, 1.64, 2.7\times10^{-5}, 0.669)$ and $c\in [-1,1]$.}
    \label{fig::bifurcation}
\end{figure}
\begin{figure}
    \centering
    \includegraphics[scale=0.32]{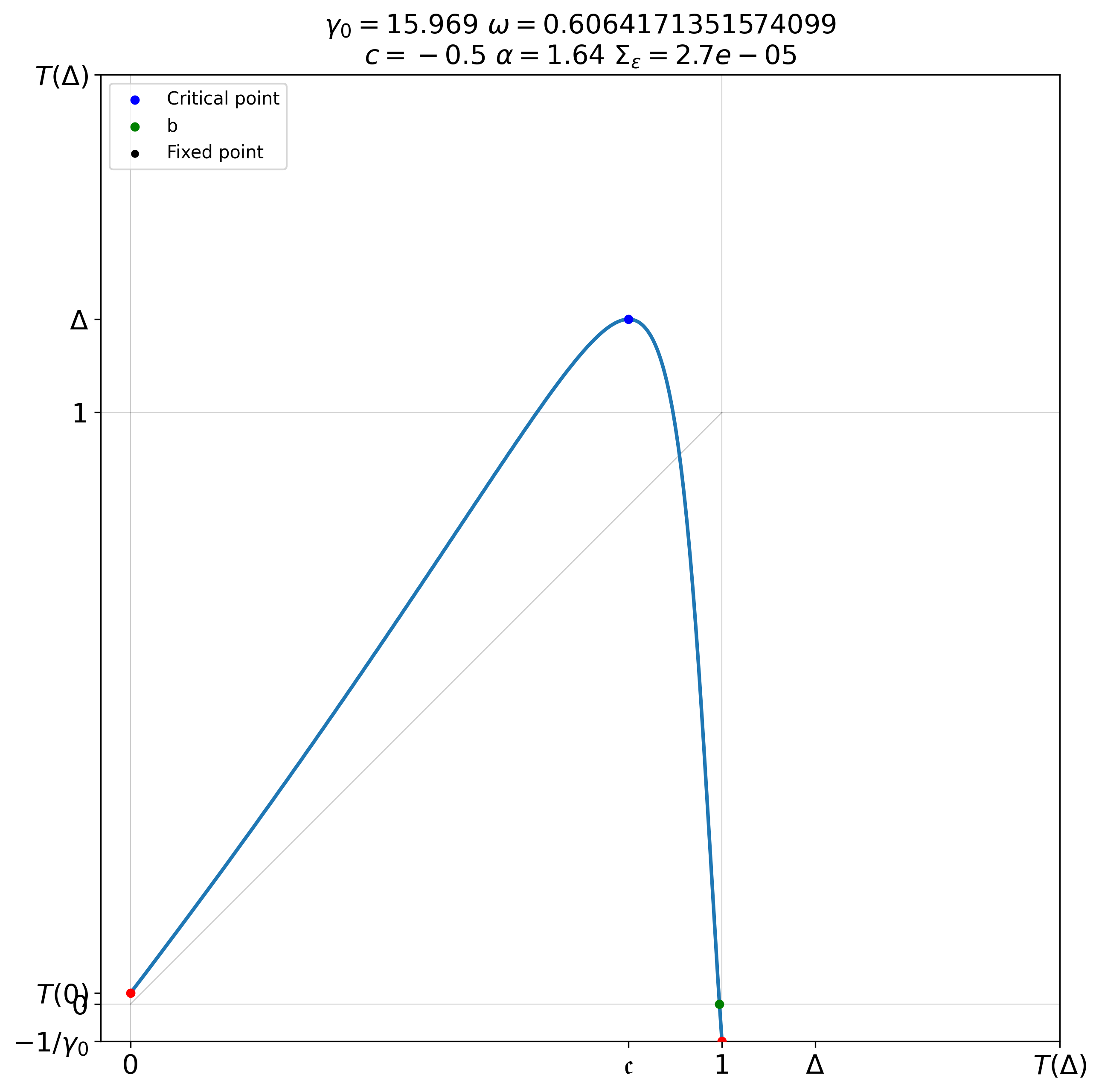}
    \includegraphics[scale=0.32]{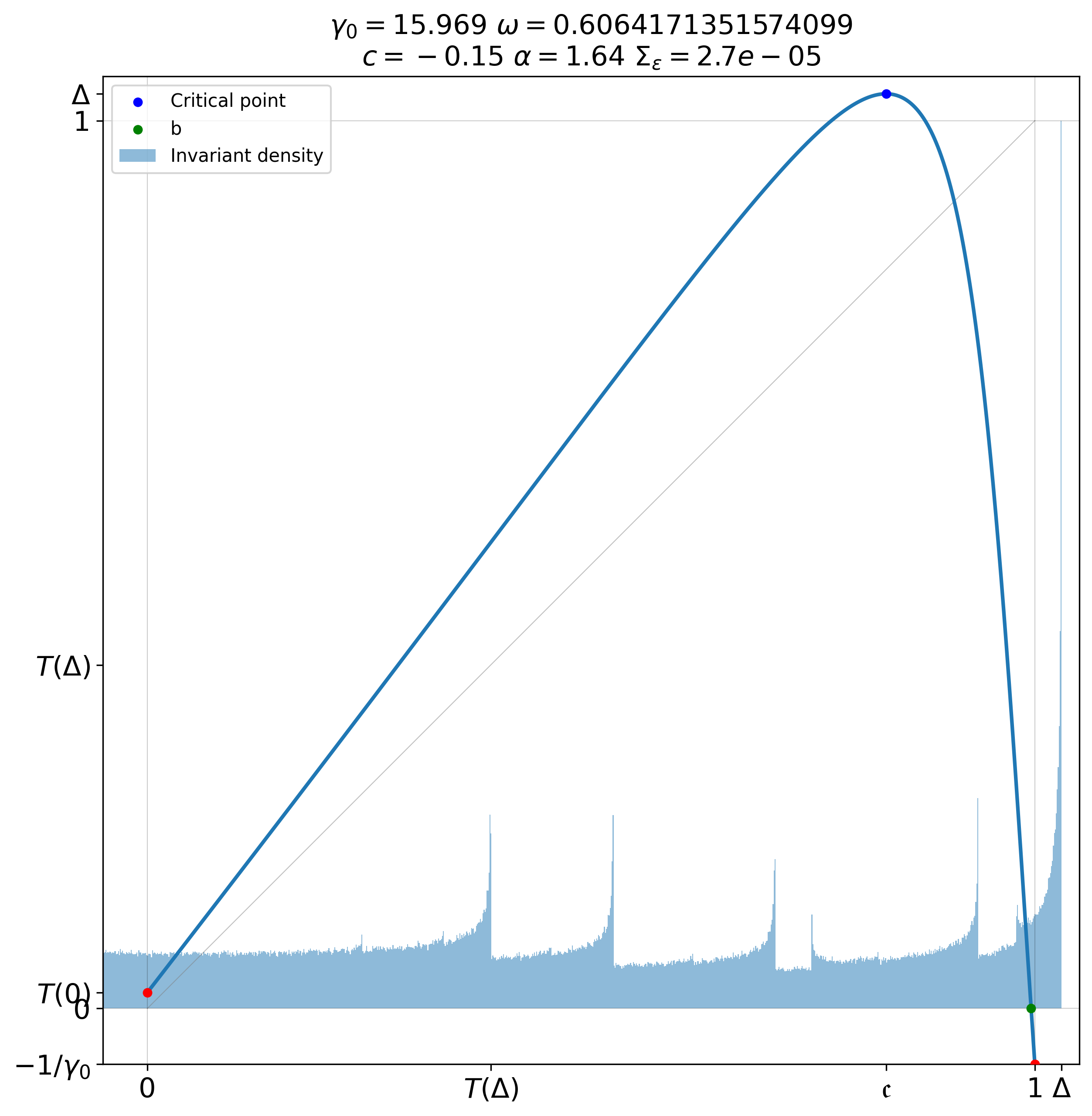}
    \includegraphics[scale=0.32]{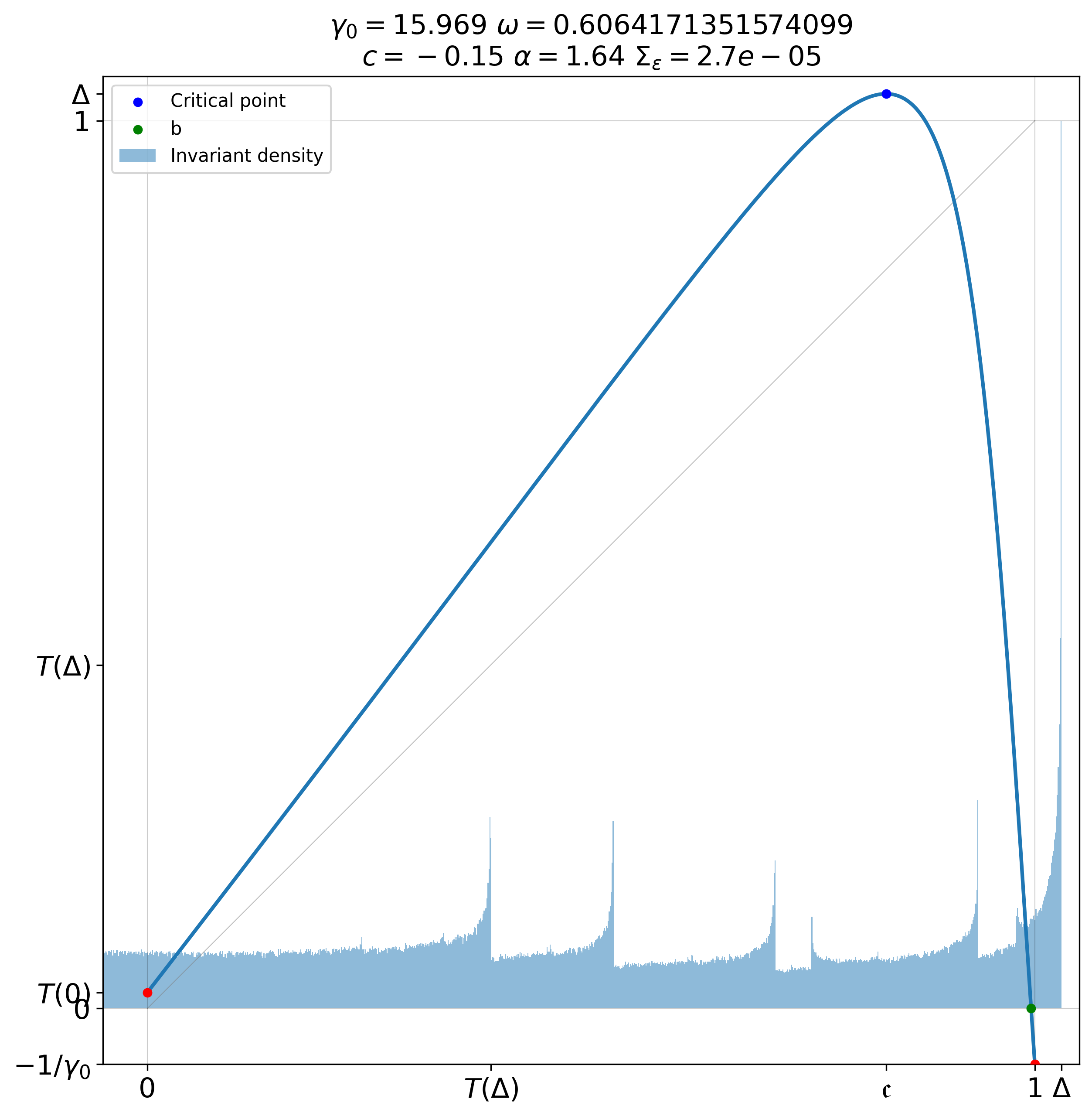}
    \includegraphics[scale=0.32]{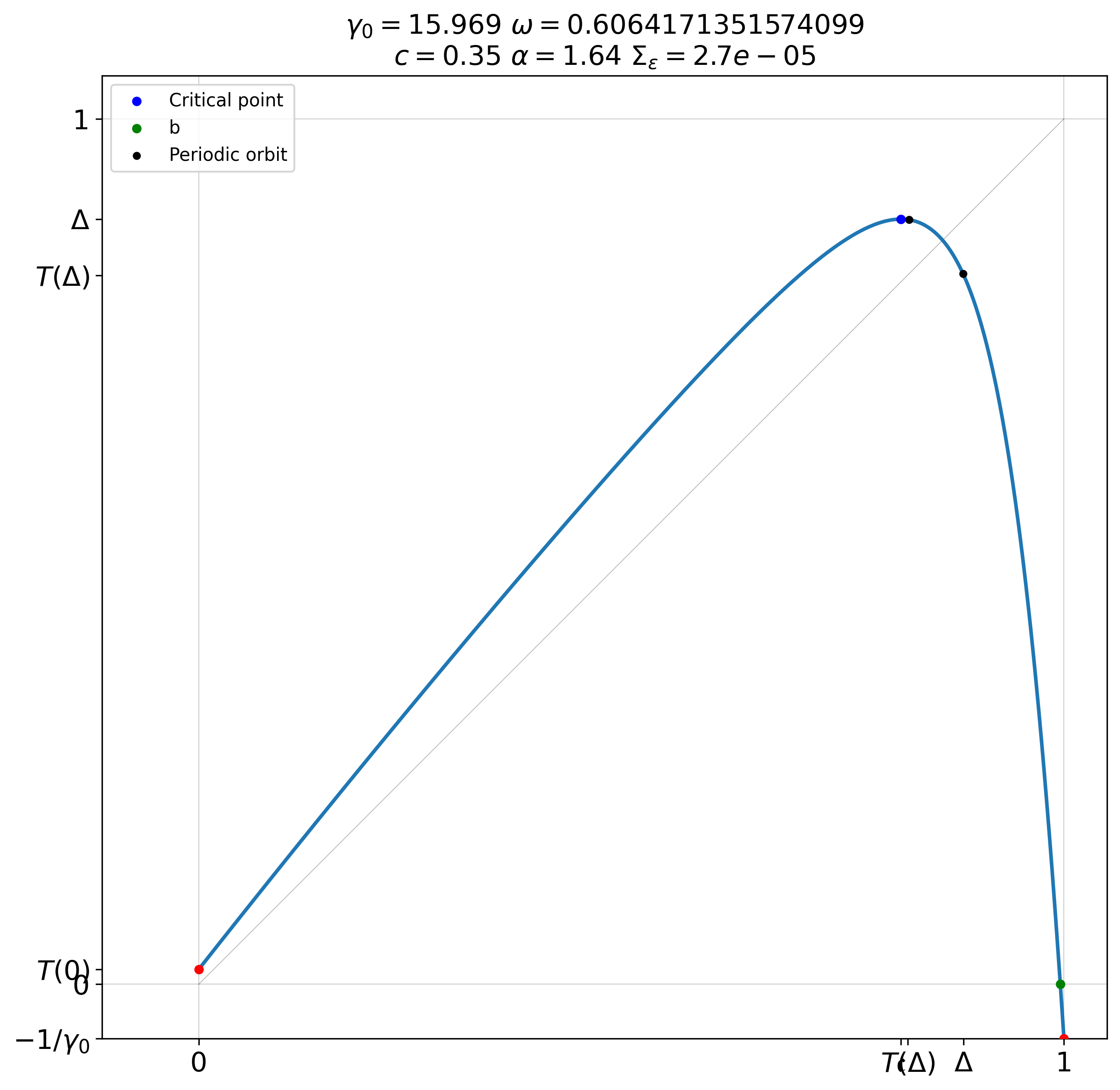}
    \caption{Plot of the deterministic component $T(\phi)$. \textit{Parameters' configuration:} $(\gamma_0, \alpha, \Sigma_{\epsilon}, \omega)=(15.969, 1.64, 2.7 \times 10^{-5}, 0.669)$. The specific value of $c$ is reported in the title of each panel.}
    \label{fig::framebyframebifurcation}
\end{figure}

To identify more precisely the signature of a chaotic behaviour, we compute the Lyapunov exponent as a function of $c$. For the deterministic map, the Lyapunov exponent is positive if and only if $T$ admits an absolutely continuous invariant measure. Figure \ref{fig::lyapunov}, from top to bottom, shows the estimated Lyapunov exponent for the deterministic map, as well as for the random system for different intensities of the noise. The Lyapunov exponent is not displayed for some values of the parameter $c$ because of some numerical issues we encountered to determine the intersection between the map and the horizontal axis. For this reason, it is not possible to fully appreciate that the exponent becomes a smooth function of $c$ when add even a small amount of noise, in agreement with Theorem \ref{thm::continuityLyapunov}. Figure \ref{fig::lyapunov} shows also the validity of Proposition \ref{cly} and therefore indirectly of Assumption \ref{itm:Ap}.

\begin{figure}
    \centering
    \includegraphics[scale=0.52]{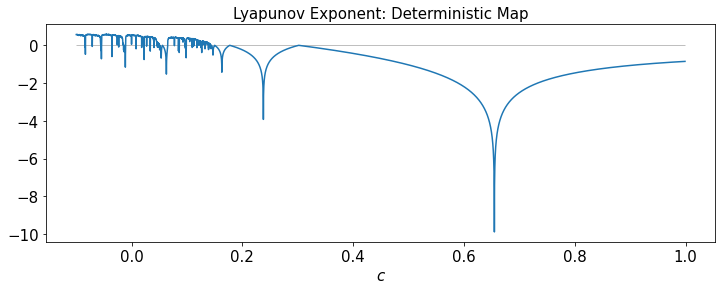}
    \includegraphics[scale=0.52]{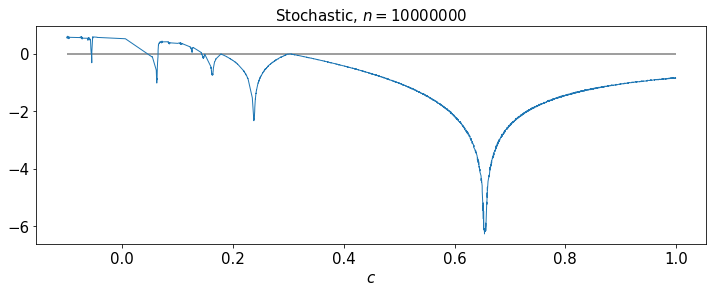}
    \includegraphics[scale=0.52]{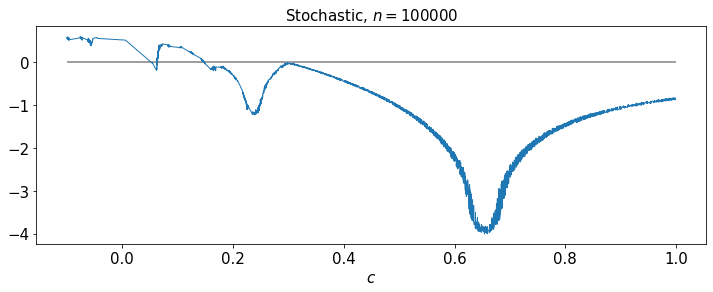}
    \includegraphics[scale=0.52]{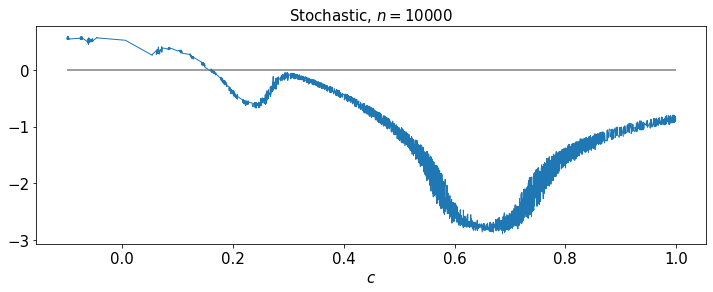}
    \caption{Lyapunov exponent for deterministic and stochastic maps. \textit{Parameters' configuration:} $(\gamma_0, \alpha, \Sigma_{\epsilon}, \omega)=(15.969, 1.64, 2.7 \times 10^{-5}, 0.669)$.}
    \label{fig::lyapunov}
\end{figure}

 Finally, Figure \ref{fig::random} displays the random map in Equation \eqref{eq::randomtransformations} together with the quantiles of the distribution of the graphs of the maps associated with the random maps. Notice that we use a different set for the parameters to emphasize the effect of the noise.
\begin{figure}
\hspace{-1.0cm}
    \centering
    \includegraphics[scale=0.50]{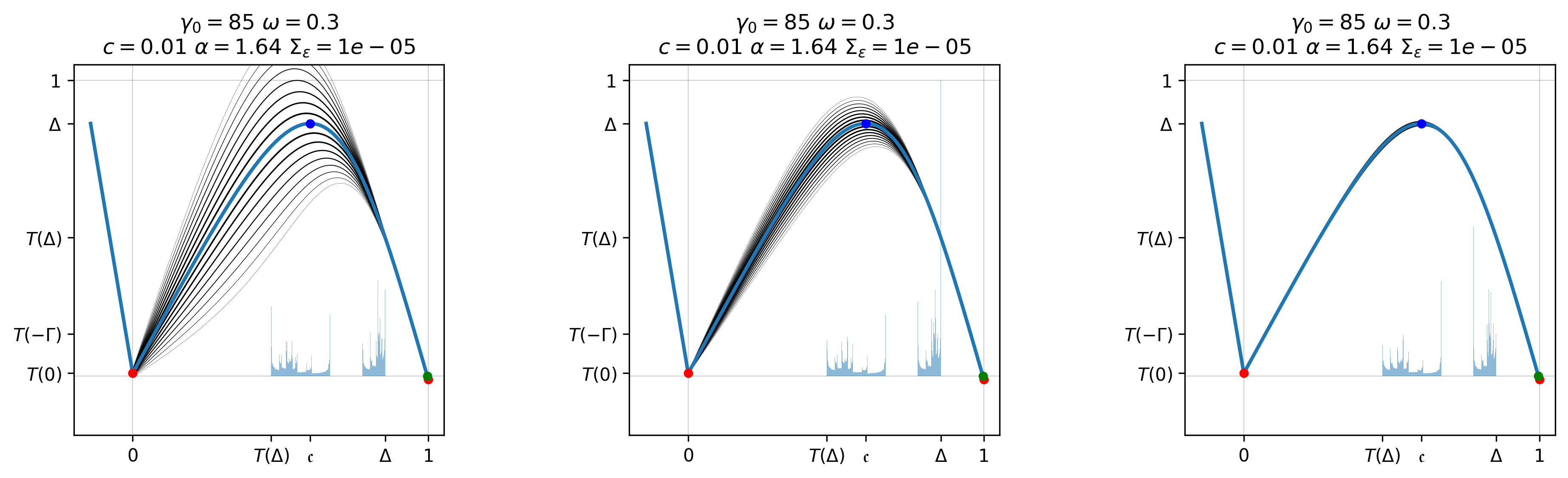}
    \caption{Random maps in Equation \eqref{eq::randomtransformations} together with the quantiles of the distribution of the graphs of the maps associated with the random maps; $\n=10, 100, 10000$.}
    \label{fig::random}
\end{figure}

\subsection{Limit theorems}\label{subsec::limittheorem}
We here investigate the validity of the Central Limit Theorem in Theorem \ref{thm::centrallimittheorem}-($e2$). We proceed in the following way. First, we choose as function $g \in BV$ such that $\int_{\mathbb{R}}g\,d\mu_{\n}=0$ the function
\begin{equation*}
    g(x) = \sin(x) - \int_{\mathbb{R}}\sin(t)\,dt.
\end{equation*}
Notice that in principle we would like to have a function $g$ with null average with respect the unknown measure $\mu_{\n}$; the function in the previous equation verifies this property with respect the Lebesgue measure. Nonetheless, we verify numerically the validity of the cited property also for $\mu_{\n}$. Then, we generate 20,000 orbits of length 10,000 by using the random transformation. In this way, for each $t \in \{1,\ldots, 10000\}$ we have a sample of the quantity $S_t$ in Equation \eqref{eq::sumslimittheorem}. Therefore, we can test if the distribution of $\frac{S_t}{\sqrt{t}}$ becomes more and more Gaussian as $t$ increases. In order to do so, we apply three normality tests, namely the Shapiro (\cite{shaphiro1965analysis}), the normal test of D'Agostino and Pearson's (\cite{diagostino1971omnibus, d1973tests}), and the Jarque-Bera's test (\cite{jarque1980efficient}). They all tests the null hypothesis that a sample comes from a normal distribution. Table \ref{tab::normality} reports the results. Within each row, the two subrows are the value of the test and, between brackets, the $p$-value From the table it is clear that the distribution of $\frac{S_t}{\sqrt{t}}$ becomes more and more Gaussian as $t$ increases, confirming the Central Limit Theorem stated above.        

\begin{table}[h]
\centering
\begin{tabular}{|c|c|c|c|c|}
\hline
\multirow{2}{*}{Normality Test}&\multicolumn{4}{c|}{$t$}\\
\cline{2-5}
 & $10$ & $1000$ & $5000$ & $10000$\\
\hline
\multirow{2}{*}{Shapiro}&0.952 &0.968 &\textbf{0.997} &\textbf{0.998}\\
                        &$(1.55\times10^{-17})$ &$(5.7\times10^{-14})$ &\textbf{(0.11)} &\textbf{(0.86)}\\
\hline
\multirow{2}{*}{Normal Test}&714.16 &67.55 &\textbf{4.19} &\textbf{0.282}\\
                        &$(8.34\times10^{-156})$ &$(2.15\times10^{-15})$ &\textbf{(0.12)} &\textbf{(0.86)} \\
\hline
\multirow{2}{*}{Jarque-Bera}&61.83 &79.73 &\textbf{4.26} &\textbf{0.33}\\
                            &$(3.7\times10^{-14})$ &$(4.85\times10^{-18}$ &\textbf{(0.11)} &\textbf{(0.84)} \\     
\hline
\end{tabular}
\caption{Normality tests for the variable $\left(\frac{S_t}{\sqrt{t}}\right)$ for different values of $t$. Each row reports the value af the tests and, between parentheses, the $p$-value.}
\label{tab::normality}
\end{table}

\subsection{Multifractal Analysis}\label{subsec::multifractal}
In this subsection, we compute the spectrum of the generalized dimension $D_q$ as in Subsection \ref{subsec::multifractal} by combining Equation \eqref{eq::generalizeddimension} with Equations \eqref{eq::limit} and \eqref{eq::limit2}. The results are displayed in Figure \ref{fig::generalizedOne}. The \textit{gray line} represents the value for the comparison as computed in Equation \eqref{eq::generalizedexplicit}. In order to compute the other lines we proceed in the following way. For $q>1$ we approximate the integral in Equation \eqref{eq::limit} by considering the so called \textit{partition sums} 
\begin{equation*}
    Z_{r}(q) = \sum_{\mu(B) \neq 0} (\mu(B))^{q},   
\end{equation*}
where the sum runs over all intervals $B$ of size $r$. In particular, we follow, e.g., \cite{riedi1995improved} and we restrict the variable $r$ to a sequence $r_n$ in order to give meaningful results. Our values of $r$ are defined by: $r = \text{np.linspace}(0.5\times 10^{-5}, 10^{-5},100)$. For a fixed $r$, the occupation number $n_{i}(r)$ of the $i$-th interval is defined as the number of sample points it contains out of $N$ sample points from the trajectory of our dynamical system. The measure $\mu_i$ of the interval $B_i$ is the fraction of time which a generic trajectory on the attractor spends in the $i$-th interval $B_i$  and is roughly equal to $n_i(r)/N$. Therefore, we compute $D(q)$ as the slope of a linear fit of 
\begin{equation*}
\log Z_r(q) = \log \left(\sum_{i} (n_i(r))^q\right)
\end{equation*}
against $\log r$; note that we have dropped the normalization factor $N = \sum_{i} n_i(r)$ since it is independent of $r$. The computation of $D(1)$ follows the same logic. Instead, for negative $q$ we follow \cite{riedi1996numerical} and we replace the occupation numbers $n_i(r)$ by the extended occupation numbers $n^{*}_i(r)$ which are defined by
\begin{equation*}
    n_i^{*}(r)=\sum_{j\,:\,B_j \subset B_{i}^{*}}n_i(r),
\end{equation*}
that is the number of sample points contained in the interval $B_i$ and its neighboring boxes. By looking, again, at Figure \ref{fig::generalizedOne}, it is interesting to see that for $\n=10$ the quantity $D_q$ for the random maps becomes almost a constant equals to 0.8--0.9, whereas this quantity changes more for higher values of $\n$ and, in particular, for the unperturbed map (proxied by $\n=10^{15}$). In particular, as we suspected in section \ref{subsec::multifractal}, the generalized dimension $D_q$ becomes more and more pronounced when $\n$ grows as the stationary measure converges (weakly) to the invariant measure of the deterministic map. In this respect, we think that the previous multifractal analysis could help us in discriminating  between chaotic and random behaviors. Indeed, the invariant measure of a deterministic map has usually fine properties which reveal themselves in a fractal or multifractal structure of the density. For our unimodal map $T$ this is due to the presence of countably many singularities for the density $h$. Instead, the equilibrium measure of Markov chains are usually  more {\em uniform} and indistinguishable  from absolutely continuous measures with bounded densities. Notice that the curves in Figure \ref{fig::generalizedOne} resemble the curves in Figure 3 in \cite{yan2021multiscale}, where authors investigate the multifractal features of liquidity in China's stock market. They claim that the liquidity time series at the studied time is no longer subject to a standard random walk process, but subject to a fractal biased random walk process. This shows that, theoretically, it is feasible to predict the liquidity of the Chinese securities market.

\begin{figure}
    \centering
    \includegraphics[scale=0.8]{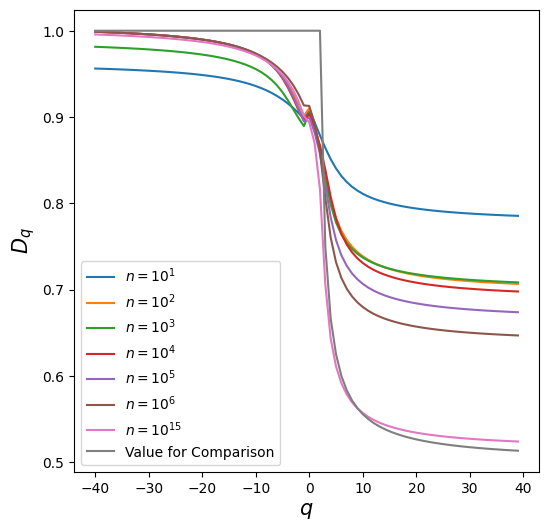}
    \caption{Spectrum of the generalized dimension $D_q$, for different values of the noise level $\n$. The gray line is derived from Eq. \eqref{eq::generalizedexplicit}. Notice that for our unimodal map the picture shows the order of divergence as the square root of the singular points as for the Benedicks-Carleson type.}
    \label{fig::generalizedOne}
\end{figure}

\subsection{Extreme Value Theory}\label{subsec::evt}
Finally we verify the validity of Proposition \ref{prop:extremevalue} on EVT. We proceed in the following way. We fix values for the parameters characterizing the map $(\gamma_0, \omega, c, \alpha, \Sigma_{\epsilon})$ as described above, for the initial point of each orbit $x_0$ and $z$ chosen randomly in the dynamical core ($x_0=0.38$ and $z=0.80$), for the parameter $\tau$ ($\tau=\log(10)$), and for the intensity of the noise ($\n=10^{3}$). Then, we determine numerically the sequence $u_t$ in such a way that $\mu_{\n}(B(z,e^{-u_t})) = \frac{\tau}{t}$, where $\mu_{\n}$ estimated from the histogram constructed with a very long orbit. The \textit{Left Panel} of Figure \ref{fig::evtempirical} shows the sequence $u_t$ as a function of $t$ and the  \textit{Right Panel}  displays the estimated $\mathbb{P}_{\n}(M_t^{(\n)} \leq u_t)$, where $M_t^{(\n)}$ is defined in Equation \eqref{eq::randomvariableEVT}, as a function of $t$, together with the theoretical value $e^{-\tau}$ (\textit{Red horizontal line}). The estimated probability  converges to the theoretical value, confirming our EVT results. From a financial point of view, this means that we are able to compute what is the probability that, given an initial leverage, the first time the leverage is ``close" to a given target is larger than $t$.\\
\begin{figure}
    \centering
    \includegraphics[scale=0.6]{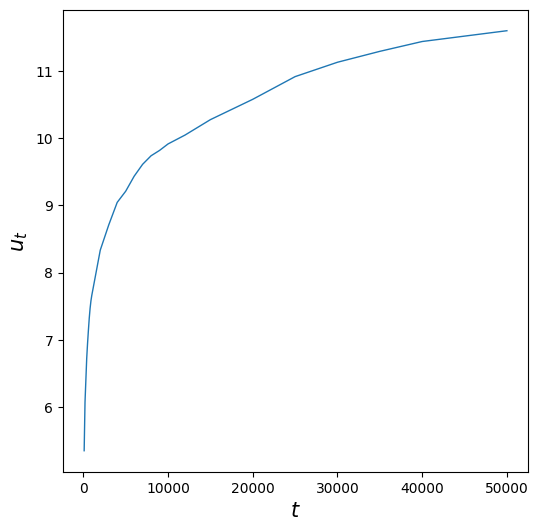}\\
    \includegraphics[scale=0.6]{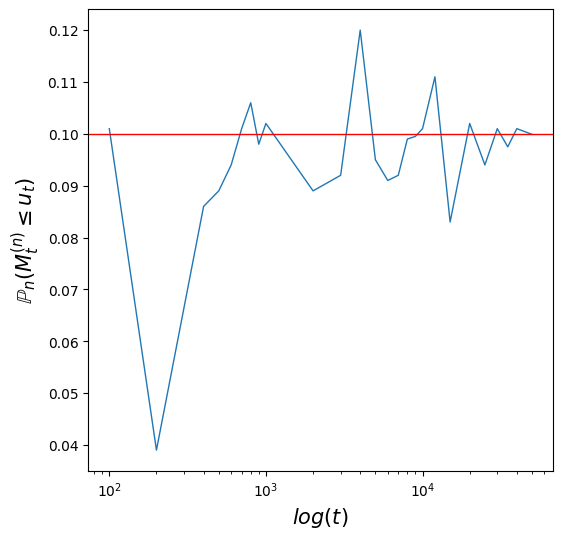}
    \caption{\textit{Top panel}: the sequence $u_t$ as a function of $t$; \textit{Bottom panel}: estimated $\mathbb{P}_{\n}(M_t^{(\n)} \leq u_t)$, where $M_t^{(\n)}$ is defined in Equation \eqref{eq::randomvariableEVT}, as a function of $\log t$, together with the theoretical value $e^{-\tau}$ (\textit{Red horizontal line})}
    \label{fig::evtempirical}
\end{figure}

\newpage
\paragraph{Acknowledgements}
SV thanks the Mathematical Research Institute MATRIX,
the Sydney Mathematical Research Institute (SMRI), the University of New South Wales, and the University of Queensland for their support and hospitality and where part of this research was performed.

\paragraph{Funding}
This research was supported by the research project `Dynamics and Information Research Institute - Quantum Information, Quantum Technologies' within the agreement between UniCredit Bank and Scuola Normale Superiore. F.L.\ and S.M.\ acknowledge partial support by the European Program scheme `INFRAIA-01-2018-2019: Research and Innovation action', grant agreement \#871042 'SoBigData++: European Integrated Infrastructure for Social Mining and Big Data Analytics'. The research of SV was supported by the project {\em Dynamics and Information Research Institute} within the agreement between UniCredit Bank and Scuola Normale Superiore di Pisa and by the Laboratoire International Associ\'e LIA LYSM, of the French CNRS and  INdAM (Italy). SV was also supported by the project MATHAmSud TOMCAT 22-Math-10, N. 49958WH, du french  CNRS and MEAE.

\paragraph{Data Availability} 
The authors will provide the code to reproduce the results of the article upon request. 

\paragraph{Declarations}
The authors have no relevant financial or non-financial interests to disclose.
\newpage
\printbibliography
\end{document}